\documentclass[11pt]{article}
\usepackage[position=bottom]{subfig}
\usepackage{amsmath, amsthm, amssymb}
\usepackage[font=footnotesize]{caption}
\usepackage{tkz-graph}
\usepackage{marginnote}
\usepackage{verbatim}
\usepackage[top=1.0in, bottom=1.0in, left=1.0in, right=1.0in]{geometry}
\usepackage{color}
\usepackage{hyperref}

\usepackage{sectsty}
\allsectionsfont{\sffamily}

\theoremstyle{plain}
\newtheorem{thm}{Theorem}
\newtheorem*{VAL}{Vizing's Adjacency Lemma (VAL)}
\newtheorem*{mainthm}{Main Theorem}

\newtheorem{lem}[thm]{Lemma}

\newtheorem*{HZconj}{Hilton--Zhao Conjecture}

\newcommand{\card}[1]{\left|#1\right|}
\newcommand{\floor}[1]{\lfloor#1\rfloor}

\def\adj{\leftrightarrow}

\def\vph{\varphi}
\def\hat{\widehat}
\newcommand\G{\mathcal{G}}
\newcommand\HH{\mathcal{H}}

\RequirePackage{marginnote,hyperref}
\newcommand{\aside}[1]{\marginnote{\scriptsize{#1}}[0cm]}
\newcommand{\aaside}[2]{\marginnote{\scriptsize{#1}}[#2]}
\newcommand\Emph[1]{\emph{#1}\aside{#1}}
\newcommand\EmphE[2]{\emph{#1}\aaside{#1}{#2}}

\newcommand\bs{\boldsymbol}

\begin{document}
\author{Daniel W. Cranston\thanks{Department of Mathematics and Applied
Mathematics, Viriginia Commonwealth University, Richmond, VA;
\texttt{dcranston@vcu.edu}; 
The first author's research is partially supported by NSA Grant
H98230-15-1-0013.}
\and
Landon Rabern\thanks{
Franklin \& Marshall College, Lancaster, PA;
\texttt{landon.rabern@gmail.com}}
}
\title{The Hilton--Zhao Conjecture is True\\ for Graphs with Maximum Degree 4}
\maketitle
\begin{abstract}
A simple graph $G$ is \emph{overfull} if $\card{E(G)}>\Delta\floor{\card{V(G)}/2}$.
By the pigeonhole principle, every overfull graph $G$ has $\chi'(G)>\Delta$.
The \emph{core} of a graph, denoted $G_\Delta$, is the subgraph
induced by its vertices of degree $\Delta$.  Vizing's Adjacency Lemma implies
that if $\chi'(G)>\Delta$, then $G_\Delta$ contains cycles.  Hilton and Zhao
conjectured that if $G$ is connected with $\Delta\ge 4$ and $G_\Delta$ has
maximum degree 2, then $\chi'(G)>\Delta$ precisely when $G$ is overfull.
We prove this conjecture for the case $\Delta=4$.
\end{abstract}
\date{}

\section{Introduction and Proof Outline}
A \EmphE{proper edge-coloring}{-2mm} of a graph $G$ assigns colors to its edges so
that edges receive distinct colors whenever they share an endpoint.  The
\EmphE{edge-chromatic number}{4mm} \emph{of $G$}, denoted $\chi'(G)$, is the smallest number
of colors that allows a proper edge-coloring of $G$.  Vizing showed that always
$\chi'(G)\le\Delta(G)+1$, where $\Delta(G)$ denotes the maximum degree of $G$.
(In this paper, all graphs are \EmphE{simple}{6mm}, which means that every pair of
vertices is joined by either 0 or 1 edges.)  Since always $\chi'(G)\ge
\Delta(G)$, we call a graph \EmphE{class 1}{1mm} when $\chi'(G)=\Delta(G)$ and call
it \EmphE{class 2}{5mm} when $\chi'(G)=\Delta(G)+1$.
For brevity, in what follows we write $\Delta$ to denote $\Delta(G)$, whenever
the context is clear.

Erd\H{o}s and Wilson~\cite{EW77} showed that almost every graph is class 1.  In
contrast, Holyer~\cite{holyer} showed that it is NP-hard to determine whether a
graph is class 1 or class 2.  As a result, most work in this area focuses on
proving sufficient conditions for a graph to be either class 1 or class 2.  
A \emph{$k$-vertex}\aside{$k/k^-$-vertex/ neighbor} is one of degree $k$, and a
\emph{$k^-$-vertex} is one of degree at most $k$.  A \emph{$k$-neighbor} (and
\emph{$k^-$-neighbor}) of a vertex $v$ is defined analogously.  
If $G$ is class 1, then $|E(G)|\le \Delta\floor{\card{V(G)}/2}$.  
This observation motivates the following definition.
A graph $G$ is \Emph{overfull} if $\card{E(G)} >
\Delta\floor{\card{V(G)}/2}$.  
Every overfull graph is class
2, since it has more edges than can appear in $\Delta$ color classes.
A graph $G$ is \Emph{critical} if $\chi'(G)>\Delta$ and
$\chi'(H)\le \Delta$ for every proper subgraph $H$.  An edge $e\in E(G)$ is a
\Emph{critical edge} of $G$ if $\chi'(G-e)<\chi'(G)$.
It is easy to show that
every class 2 graph $G$ contains a critical subgraph $H$ with the same maximum
degree as $G$.  Critical graphs are useful because they have more
structure than general graphs.  For example, Vizing proved the following.

\begin{VAL} 
Let $G$ be a class 2 graph with maximum degree $\Delta$.  If $vw$ is a
critical edge of $G$, then $w$ has at
least $\max\{\Delta+1-d(v),2\}$ $\Delta$-neighbors.
\end{VAL}

The \Emph{core} of a graph $G$, denoted $G_\Delta$, is the subgraph of
$G$ induced by $\Delta$-vertices.  VAL shows that if $G$ is class 2, then
$G_\Delta$ must contain cycles (this was also proved by
Fournier~\cite{fournier}).  So a natural question is which class 2
graphs have a core consisting of disjoint cycles.  Hilton and Zhao
\cite{HZ96} conjectured exactly when this happens.  Let $P^*$ denote
the Petersen graph with a vertex deleted. 

\tikzstyle{majorStyle}=[shape = circle, minimum size = 6pt, inner sep = 2.2pt, draw]
\tikzstyle{major}=[shape = circle, minimum size = 6pt, inner sep = 2.2pt, draw]
\tikzstyle{minorStyle}=[shape = rectangle, minimum size = 6pt, inner sep = 2.2pt, draw]
\tikzstyle{minor}=[shape = rectangle, minimum size = 6pt, inner sep = 2.2pt, draw]
\tikzstyle{labeledStyle}=[shape = rectangle, minimum size = 6pt, inner sep = 2.2pt, draw]
\tikzstyle{VertexStyle} = []
\tikzstyle{EdgeStyle} = []
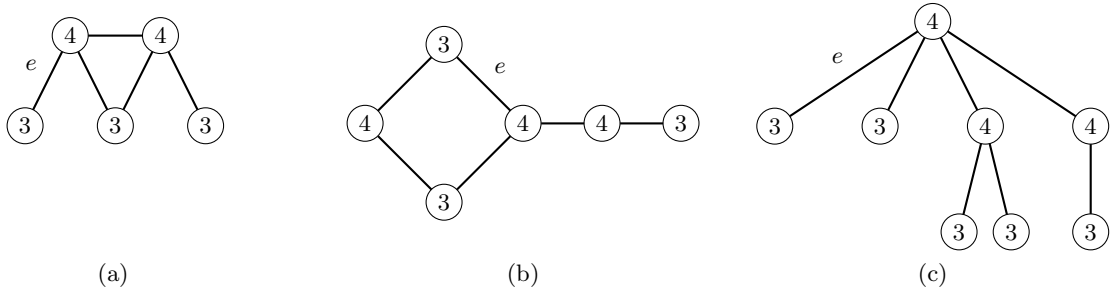
\begin{figure}[!b]
\renewcommand{\ttdefault}{ptm}
\begin{center}
\subfloat[]{\makebox[.33\textwidth]{
\begin{tikzpicture}[scale = 6]
\Vertex[style = major, x = 0.25, y = 0.75, L = \footnotesize {\textrm{3}}]{v0}
\Vertex[style = major, x = 0.35, y = 0.95, L = \footnotesize {\textrm{4}}]{v1}
\Vertex[style = major, x = 0.45, y = 0.75, L = \footnotesize {\textrm{3}}]{v2}
\Vertex[style = major, x = 0.55, y = 0.95, L = \footnotesize {\textrm{4}}]{v3}
\Vertex[style = major, x = 0.65, y = 0.75, L = \footnotesize {\textrm{3}}]{v4}
\Edge[label= \footnotesize {$e$}, , labelstyle={auto=left, fill=none}](v0)(v1)
\Edge[](v1)(v3)
\Edge[](v1)(v2)
\Edge[](v2)(v3)
\Edge[](v4)(v3)
\draw[white] (.20,.475)--(.20,0.49);
\end{tikzpicture}
}}
\subfloat[]{\makebox[.33\textwidth]{
\begin{tikzpicture}[rotate=90,scale=7]
\Vertex[style = major, x = 0.35, y = 0.80, L = \footnotesize {\textrm{3}}]{v0}
\Vertex[style = major, x = 0.65, y = 0.80, L = \footnotesize {\textrm{3}}]{v1}
\Vertex[style = major, x = 0.50, y = 0.65, L = \footnotesize {\textrm{4}}]{v2}
\Vertex[style = major, x = 0.50, y = 0.95, L = \footnotesize {\textrm{4}}]{v3}
\Vertex[style = major, x = 0.50, y = 0.50, L = \footnotesize {\textrm{4}}]{v4}
\Vertex[style = major, x = 0.50, y = 0.35, L = \footnotesize {\textrm{3}}]{v5}
\Edge[](v2)(v0)
\Edge[label= \footnotesize {$e$}, , labelstyle={auto=right, fill=none}](v2)(v1)
\Edge[](v3)(v0)
\Edge[](v3)(v1)
\Edge[](v4)(v2)
\Edge[](v5)(v4)
\draw[white] (.26,.55)--(.27,.55);
\end{tikzpicture}
}}
\subfloat[]{\makebox[.33\textwidth]{
\begin{tikzpicture}[scale = 7]
\Vertex[style = major, x = 0.5, y = 0.849, L = \footnotesize {\textrm{4}}]{v0}
\Vertex[style = major, x = 0.200, y = 0.650, L = \footnotesize {\textrm{3}}]{v1}
\Vertex[style = major, x = 0.400, y = 0.650, L = \footnotesize {\textrm{3}}]{v2}
\Vertex[style = major, x = 0.600, y = 0.650, L = \footnotesize {\textrm{4}}]{v3}
\Vertex[style = major, x = 0.800, y = 0.650, L = \footnotesize {\textrm{4}}]{v4}
\Vertex[style = major, x = 0.550, y = 0.449, L = \footnotesize {\textrm{3}}]{v6}
\Vertex[style = major, x = 0.649, y = 0.449, L = \footnotesize {\textrm{3}}]{v5}
\Vertex[style = major, x = 0.800, y = 0.449, L = \footnotesize {\textrm{3}}]{v7}
\Edge[label= \footnotesize {$e$}, , labelstyle={auto=right, fill=none}](v0)(v1)
\Edge[](v2)(v0)
\Edge[](v3)(v0)
\Edge[](v4)(v0)
\Edge[](v5)(v3)
\Edge[](v6)(v3)
\Edge[](v7)(v4)
\end{tikzpicture}
}}
%
\end{center}
\captionsetup{
  margin = 90pt
}
\caption{Each configuration cannot appear in a class 2 graph\\ in $\G_4$.
(The number at each vertex specifies its degree in $G$.)
\label{fig:3reducibles}}
\end{figure}

\begin{HZconj}
If $G$ is a connected graph with $\Delta\ge 3$ and with core of maximum degree
at most 2, then $G$ is class 2 if and only if $G$ is $P^*$ or $G$ is overfull.
\end{HZconj}

David and Gianfranco Cariolaro~\cite{cariolaro03} proved this
conjecture when $\Delta=3$.  Kral', Sereni, and
Stiebitz~\cite[p.~57--63]{SSTF-book} gave an alternate proof.
An easy counting argument shows that every graph satisfying the hypotheses of
the conjecture has average degree at most $\Delta-1+\frac{\Delta-1}{2\Delta-3}$;
when $\Delta=3$, this is $\frac83$.  By Lemma~\ref{lem0} below, any
counterexample to the conjecture must be critical.  Thus, the case $\Delta=3$
is also implied by our result~\cite{CRfixer3} that every critical graph with
$\Delta=3$ (other than the Petersen graph with a vertex deleted) has average
degree at least $\frac{46}{17}\approx2.706$.

In this paper, we prove the conjecture when $\Delta=4$. 
Let $\mathcal{G}_k$ denote the class of connected graphs with maximum degree
$k$ in which the core has maximum degree at most 2; that is, each $k$-vertex 
has at most two $k$-neighbors.  Let $\HH_k$ denote the class of connected graphs $G$ such
that (i) $G$ has maximum degree $k$, (ii) $G$ has minimum degree $k-1$, (iii)
$G_\Delta$ is a disjoint union of cycles, and (iv) every vertex of $G$ has a
$\Delta$-neighbor.  Note that $\HH_k\subseteq \G_k$.
To prove our main result, we use a
lemma of Hilton and Zhao~\cite{HZ92}, implied by VAL.
To keep this paper self-contained, we include a proof.

\begin{lem}
If $G\in \G_k$ with $k\ge3$ and $\chi'(G)>k$, then $G\in \HH_k$ and $G$ is
critical.
\label{HZlem}
\label{lem0}
\end{lem}
\begin{proof}
Let $G$ satisfy the hypotheses and let $H$ be a $k$-critical subgraph of $G$.
Suppose $H$ has a $(k-2)^-$-vertex $v$.  By VAL, $v$ has a $k$-neighbor
$w$.  Now $w$ has at least $k+1-d(v)\ge k+1-(k-2)=3$ neighbors of degree $k$, a
contradiction, since $H\in \G_k$.  Thus, $H$ has no $(k-2)^-$-vertex.

Suppose that $V(H)\subsetneq V(G)$.  Choose $v\in V(H)$ and $w\in
V(G)\setminus V(H)$ such that $w\in N_G(v)$.  If $d_G(v)\le k-1$, then
$d_H(v)\le k-2$, a contradiction.  So $d_G(v)=k$.  Since $H$ is critical, $v$
has a $k$-neighbor $w$.  But now $w$ has at most two $k$-neighbors in $G$ (since
$G\in G_k$), one of which is $v$.  So $w$ has at most one $k$-neighbor in $H$,
contradicting VAL.  Thus, $V(H)=V(G)$.  Finally, suppose there exists $e\in
E(G)\setminus E(H)$.  Now either $H$ has a $(k-2)^-$-vertex or some $k$-vertex
in $H$ has at most one neighbor $w$ in $H$ with $d_H(w)=k$; both are
contradictions.  Thus, $E(G)=E(H)$.  So $G$ is critical.  Now VAL implies that
every vertex has at least two $\Delta$-neighbors. Hence, $G\in \HH_k$.
\end{proof}

Now we can prove our Main Theorem, subject to three reducibility lemmas, which we
state and prove in the next section.  In short, the lemmas say that a graph in
$\HH_4$ is class 1 whenever it contains at least one of the configurations in
Figure~\ref{fig:3reducibles} (not necessarily induced).

\begin{mainthm}
A connected graph $G$ with $\Delta = 4$ and with core of maximum degree at most
2 is class 2 if and only if $G$ is $K_5-e$.  
This implies the case $\Delta=4$ of the Hilton--Zhao Conjecture.
\end{mainthm}

\begin{proof}
Let $G$ be a graph with $\Delta=4$ and with core of maximum degree at most 2.
By Lemma~\ref{HZlem}, we assume $G\in \HH_4$.
Note that every 4-vertex in $G$ has exactly two
3-neighbors and two 4-neighbors.  Let $v$ denote a 4-vertex and let
$w_1,\ldots,w_4$ denote its neighbors, where $d(w_1)=d(w_2)=3$ and $d(w_3)=d(w_4)=4$.
When vertices $x$ and $y$ are adjacent, we write $x\adj y$.  We assume that $G$
contains no configuration in Figure~\ref{fig:3reducibles} and show that
$G$ is $K_5-e$.  
	
First suppose that $v$ has a 3-neighbor and a 4-neighbor that are adjacent.  By
symmetry, assume that $w_2\adj w_3$.  Since
Figure~\ref{fig:3reducibles}(a) is forbidden, we have $w_3\adj w_1$. 
Now consider $w_4$.  If $w_4$ has a 3-neighbor distinct from $w_1$ and $w_2$,
then we have a copy of Figure~\ref{fig:3reducibles}(b).  Hence $w_4\adj w_1$ and
$w_4\adj w_2$.  If $w_3\adj w_4$, then $G$ is $K_5-e$.  Suppose not, and let
$x$ be a 4-neighbor of $w_4$.  Since $G$ has no copy of
Figure~\ref{fig:3reducibles}(b), $x$ must be adjacent to $w_1$ and $w_2$.  This
is a contradiction, since $w_1$ and $w_2$ are 3-vertices, but now each has at
least four neighbors.  Hence, each of $w_1$ and $w_2$ is
non-adjacent to each of $w_3$ and $w_4$.
	
Now consider the 3-neighbors of $w_3$ and $w_4$.  If $w_3$ and $w_4$ have zero
or one 3-neighbors in common, then we have a copy of
Figure~\ref{fig:3reducibles}(c).  Otherwise they have two 3-neighbors in
common, so we have a copy of Figure~\ref{fig:3reducibles}(b).  
\end{proof}

We first announced the Main Theorem in~\cite{fixer-survey}, and included the
proof above.  But we did not include proofs of the reducibility lemmas that we
present in the next section.
  
\section{Reducibility Lemmas}
In this section we prove the reducibility of the three configurations in
Figure~\ref{fig:3reducibles}.  More precisely, suppose that $G\in G_4$ and $G$
contains one of these configurations, $H$, as a subgraph, not necessarily
induced (the number at each vertex of $H$ denotes its degree in $G$).  We show
that $\chi'(G)=4$.  If not, then Lemma~\ref{lem0} implies that $G\in\HH_4$ and
$G$ is critical.  Thus, $\chi'(G-e)=4$, where $e$ is the edge denoted in the
figure.  For convenience, we write \Emph{coloring} to mean edge-coloring with
colors 0, 1, 2, 3.
Since $\chi'(G-e)=4$, we begin with an arbitrary coloring $\vph$ of $G-e$.
A priori, $\vph$ could restrict to many possible colorings of $H-e$.  Starting
from $\vph$, we use repeated Kempe swaps (see below) to get a coloring of
$G-e$ that restricts to one of a few colorings of $H-e$.  We conclude
by modifying the coloring of $H-e$ to transform the coloring of $G-e$ to a
coloring of $G$.  At each step, we call the current coloring $\vph$.  So, to
change the color of some edge $xy$ to $i$, we ``let $\vph(xy)=i$''.
If $\vph$ uses color $i$ on an edge incident to vertex $v$, then $v$ \EmphE{sees
$i$}{-4mm}; otherwise $v$ \EmphE{misses $i$}{0mm}.
In the figures that follow, we typically draw all edges incident to vertices
of $H$.  However, only the edges shown in Figure~\ref{fig:3reducibles} are
considered edges of $H$; the others are \EmphE{pendant edges}{-5mm}.

An \emph{$(i,j)$-chain at a vertex $v$}\aside{$(i,j)$-chain/ linked/ unlinked} 
is the component containing $v$ of the
subgraph induced by edges colored $i$ and $j$.  If two vertices $v$ and $w$ are
in the same $(i,j)$-chain, then $v$ and $w$ are \emph{$(i,j)$-linked};
otherwise they are \emph{$(i,j)$-unlinked}.  Each $(i,j)$-chain $P$ is a
path or an even cycle.  If $P$ is a path that starts in $V(H)$, then $P$ either
ends in $V(H)$ or ends in $V(G)\setminus V(H)$.  In the latter case, $P$ \EmphE{ends
at $\infty$}{-4mm}.  To \emph{recolor}\aaside{recolor chain}{4mm} an
$(i,j)$-chain $P$ means to use color $i$ on each edge colored $j$ and vice
versa (this is typically called a Kempe swap, but here we rarely use that
term).  When $P$ contains a vertex $v$, we also say that we \EmphE{$(i,j)$-swap at
$v$}{2.5mm}.
Recoloring any chain in a coloring of $G-e$ yields another coloring
of $G-e$.  If each $(i,j)$-chain in $G-E(H)$ that starts in $V(H)$ ends at
$\infty$, then we can recolor pendant edges independently, by recoloring the
chain beginning with each pendant edge.  If, instead, an $(i,j)$-chain beginning
in $V(H)$ ends in $V(H)$, then its end edges (and endpoints) are paired, and
recoloring one edge necessarily recolors the other.  Choose $v,w\in V(H)$ that
each begin an $(i,j)$-chain in $G-E(H)$; call the chains $P_v$ and $P_w$.  If
$P_v$ and $P_w$ both end at $\infty$, the we can simulate that $P_v$ ends at
$w$, so $P_v=P_w$.  To do so, whenever we recolor $P_v$ we also recolor $P_w$.
Thus, for any pair $(i,j)\subset \{0,1,2,3\}$, we can assume that at most one
$(i,j)$-chain in $G-E(H)$ that starts in $V(H)$ ends at $\infty$.

During our process of modifying $\vph$, when we
recolor the $(i,j)$-chain $P$ at $v$, we might want to recolor $P$
but realize that this is no help if $P$ ends at $x$.  Similarly, we
might also be happy to recolor the $(i,j)$-chain $Q$ at $w$, but realize this
also is no help if $Q$ ends at $x$.  Fortunately, we can make progress, since
it is impossible for both $P$ and $Q$ to end at $x$.  To get more control when
recoloring, we frequently consider all $(i,j)$-chains in $G-E(H)$ that begin at
vertices of $V(H)$.  Now our analysis is similar, but more extensive. 
This approach is possible only when we know the color on every edge of $H-e$.
We discuss this general technique further in~\cite{fixer-survey}.

In the proofs of the reducibility lemmas, while handling one case, we often
reduce to a case handled previously.  An alternate approach would be to assume
in each case that $G-e$ has no 4-edge-coloring satisfying the hypothesis of any
previous case.  However, the approach we take has the advantage that it is more
easily adapted to give an efficient coloring algorithm.

\begin{figure}[!b]
\centering
\subfloat[]{\makebox[.48\textwidth]{
\begin{tikzpicture}[scale = 12]
\tikzstyle{VertexStyle} = []
\tikzstyle{EdgeStyle} = []
\tikzstyle{unlabeledStyle}=[shape = circle, minimum size = 6pt, inner sep = 1.2pt, draw]
\tikzstyle{labeledStyle}=[shape = circle, minimum size = 14pt, inner sep =
1.2pt, draw]

\Vertex[style = labeledStyle, x = 0.400, y = 0.800, L = \footnotesize {$v$}]{v0}
\Vertex[style = labeledStyle, x = 0.300, y = 0.650, L = \footnotesize {$z$}]{v1}
\Vertex[style = labeledStyle, x = 0.600, y = 0.800, L = \footnotesize {$w$}]{v2}
\Vertex[style = labeledStyle, x = 0.500, y = 0.650, L = \footnotesize {$x$}]{v3}
\Vertex[style = labeledStyle, x = 0.700, y = 0.650, L = \footnotesize {$y$}]{v4}
\Vertex[style = unlabeledStyle, x = 0.250, y = 0.570, L = \small {}]{v5}
\Vertex[style = unlabeledStyle, x = 0.350, y = 0.570, L = \small {}]{v6}
\Vertex[style = unlabeledStyle, x = 0.500, y = 0.570, L = \small {}]{v7}
\Vertex[style = unlabeledStyle, x = 0.350, y = 0.880, L = \small {}]{v8}
\Vertex[style = unlabeledStyle, x = 0.650, y = 0.880, L = \small {}]{v9}
\Vertex[style = unlabeledStyle, x = 0.650, y = 0.570, L = \small {}]{v10}
\Vertex[style = unlabeledStyle, x = 0.750, y = 0.570, L = \small {}]{v11}
\Edge[label = \footnotesize {$3$}, labelstyle={xshift=-.08in, fill=none}](v0)(v3)
\Edge[label = \footnotesize {$1$}, labelstyle={xshift=.08in, fill=none}](v2)(v3)
\Edge[label = \footnotesize {$2$}, labelstyle={yshift=.08in, fill=none}](v2)(v0)
\Edge[label = \footnotesize {}, labelstyle={auto=right, fill=none}](v1)(v0)
\Edge[label = \footnotesize {$3$}, labelstyle={xshift=.08in, fill=none}](v4)(v2)
\Edge[label = \footnotesize {$0$}, labelstyle={xshift=-.08in, fill=none}](v5)(v1)
\Edge[label = \footnotesize {$1$}, labelstyle={xshift=.08in, fill=none}](v6)(v1)
\Edge[label = \footnotesize {$2$}, labelstyle={xshift=.08in, fill=none}](v7)(v3)
\Edge[label = \footnotesize {$0$}, labelstyle={xshift=-.08in, fill=none}](v8)(v0)
\Edge[label = \footnotesize {$0$}, labelstyle={xshift=.08in, fill=none}](v9)(v2)
\Edge[label = \footnotesize {$1$}, labelstyle={xshift=-.08in, fill=none}](v10)(v4)
\Edge[label = \footnotesize {$2$}, labelstyle={xshift=.08in, fill=none}](v11)(v4)
\end{tikzpicture}
}}
\subfloat[]{\makebox[.48\textwidth]{
\begin{tikzpicture}[scale = 12]
\tikzstyle{VertexStyle} = []
\tikzstyle{EdgeStyle} = []
\tikzstyle{unlabeledStyle}=[shape = circle, minimum size = 6pt, inner sep = 1.2pt, draw]
\tikzstyle{labeledStyle}=[shape = circle, minimum size = 14pt, inner sep =
1.2pt, draw]
\Vertex[style = labeledStyle, x = 0.400, y = 0.800, L = \footnotesize {$v$}]{v0}
\Vertex[style = labeledStyle, x = 0.300, y = 0.650, L = \footnotesize {$z$}]{v1}
\Vertex[style = labeledStyle, x = 0.600, y = 0.800, L = \footnotesize {$w$}]{v2}
\Vertex[style = labeledStyle, x = 0.500, y = 0.650, L = \footnotesize {$x$}]{v3}
\Vertex[style = labeledStyle, x = 0.700, y = 0.650, L = \footnotesize {$y$}]{v4}
\Vertex[style = unlabeledStyle, x = 0.250, y = 0.570, L = \footnotesize {}]{v5}
\Vertex[style = unlabeledStyle, x = 0.350, y = 0.570, L = \footnotesize {}]{v6}
\Vertex[style = unlabeledStyle, x = 0.500, y = 0.570, L = \footnotesize {}]{v7}
\Vertex[style = unlabeledStyle, x = 0.350, y = 0.880, L = \footnotesize {}]{v8}
\Vertex[style = unlabeledStyle, x = 0.650, y = 0.880, L = \small {}]{v9}
\Vertex[style = unlabeledStyle, x = 0.650, y = 0.570, L = \small {}]{v10}
\Vertex[style = unlabeledStyle, x = 0.750, y = 0.570, L = \small {}]{v11}
\Edge[label = \footnotesize {$1$}, labelstyle={xshift=-.08in, fill=none}](v0)(v3)
\Edge[label = \footnotesize {$2$}, labelstyle={xshift=.08in, fill=none}](v2)(v3)
\Edge[label = \footnotesize {$0$}, labelstyle={yshift=.08in, fill=none}](v2)(v0)
\Edge[label = \footnotesize {$3$}, labelstyle={xshift=-.08in, fill=none}](v1)(v0)
\Edge[label = \footnotesize {$3$}, labelstyle={xshift=.08in, fill=none}](v4)(v2)
\Edge[label = \footnotesize {$0$}, labelstyle={xshift=-.08in, fill=none}](v5)(v1)
\Edge[label = \footnotesize {$1$}, labelstyle={xshift=.08in, fill=none}](v6)(v1)
\Edge[label = \footnotesize {$0$}, labelstyle={xshift=.08in, fill=none}](v7)(v3)
\Edge[label = \footnotesize {$2$}, labelstyle={xshift=-.08in, fill=none}](v8)(v0)
\Edge[label = \footnotesize {$1$}, labelstyle={xshift=.08in, fill=none}](v9)(v2)
\Edge[label = \footnotesize {$1$}, labelstyle={xshift=-.08in, fill=none}](v10)(v4)
\Edge[label = \footnotesize {$2$}, labelstyle={xshift=.08in, fill=none}](v11)(v4)
%
\end{tikzpicture}
}}

\caption{
(a) A coloring of $G-vz$ in Case 1.
(b) A coloring of $G$ in Case 2.\label{lem1case1}}
\end{figure}
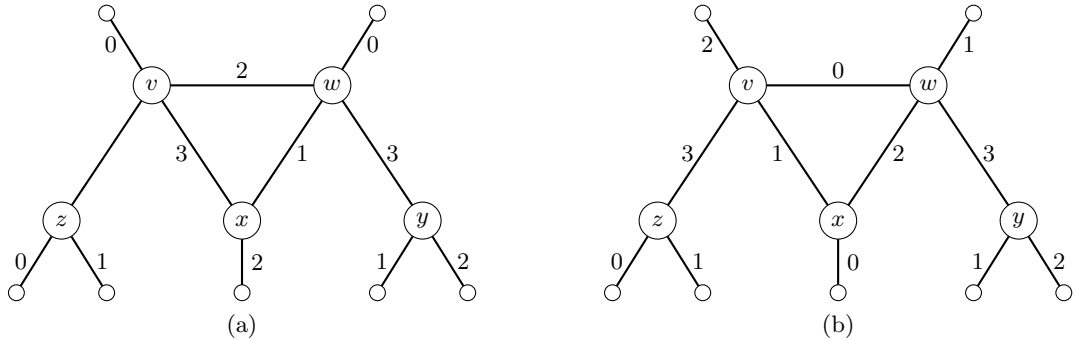

\begin{lem}
%
Suppose that $\Delta=4$ and $G$ has the configuration in
Figure~\ref{fig:3reducibles}(a) (reproduced in Figure~\ref{lem1case1}, if we
ignore the colors there). 
If $\chi'(G-vz)=4$, then $\chi'(G)=4$.
\label{lem1}
\end{lem}
\begin{proof}
We start with a coloring of $G-vz$ and assume that $G$ has no
coloring, which leads to a contradiction.  We denote by $v'$, $w'$, and $x'$ the
sole unlabeled neighbors of $v$, $w$, and $x$, respectively.
By symmetry, we assume that $z$ sees 0 and 1, and $v$ sees 0, 2, and 3. 
We repeatedly use that $v$ and $z$ must be (1,2)- and (1,3)-linked.
We consider three cases: color 0 is used on $vv'$, $vx$, or $vw$.

\textbf{Case 1: 0 is used on $\bs{vv'}$.}
Let $\vph$ be a coloring of $G-vz$, and
suppose $\vph(vv')=0$.  By symmetry, assume that $\vph(vw)=2$ and
$\vph(vx)=3$, as in Figure~\ref{lem1case1}(a).  We show that we may assume all edges are
colored as in Figure~\ref{lem1case1}(a).
Suppose $\vph(wx)=0$.  Since $v$ and
$z$ are (1,3)-linked, $\vph(xx')=1$.  Now we (1,2)-swap at $x$, which
makes $v$ and $z$ (1,3)-unlinked, a contradiction.  So $\vph(wx)=1$.  Since $v$
and $z$ are (1,2)-linked, $\vph(xx')=2$.

Suppose $\vph(wy)=0$, so $\vph(ww')=3$.  Now $y$ must see 1; otherwise a
(0,1)-swap at $y$ recolors $wx$ with 0, and $v$ and $z$ become
(1,3)-unlinked, a contradiction.  So $y$ misses 2 or 3.  Now a (1,2)-
or (1,3)-swap at $y$ makes $y$ miss 1, but nothing else has changed.  So we
are done.

So assume $\vph(wy)=3$ and $\vph(ww')=0$.  Now $y$ must see 1, since $v$ and $z$
are (1,3)-linked.  If $y$ misses 2, then a (1,2)-swap at $y$ makes $y$ miss
1, a contradiction.  So $y$ sees 2 and 1, and misses 0.  Consider the
(0,1)-chain $P$ at $y$.  $P$ must contain either (a) $w'w, wx$ or (b) $v'v$;
otherwise we (0,1)-swap at $y$ and are done.  If (a), then we (0,1)-swap at $y$
and are done, since now $v$ and $z$ are (1,3)-unlinked.  So assume (b).  Now after a
(0,1)-swap at $y$, let $\vph(vx)=0$ and $\vph(vz)=3$.
This completes Case 1.

\begin{figure}[!b]
\centering
\subfloat[]{\makebox[.48\textwidth]{
\begin{tikzpicture}[scale = 12]
\tikzstyle{VertexStyle} = []
\tikzstyle{EdgeStyle} = []
\tikzstyle{unlabeledStyle}=[shape = circle, minimum size = 6pt, inner sep = 1.2pt, draw]
\tikzstyle{labeledStyle}=[shape = circle, minimum size = 14pt, inner sep = 1.2pt, draw]
\Vertex[style = labeledStyle, x = 0.400, y = 0.800, L = \footnotesize {$v$}]{v0}
\Vertex[style = labeledStyle, x = 0.300, y = 0.650, L = \footnotesize {$z$}]{v1}
\Vertex[style = labeledStyle, x = 0.600, y = 0.800, L = \footnotesize {$w$}]{v2}
\Vertex[style = labeledStyle, x = 0.500, y = 0.650, L = \footnotesize {$x$}]{v3}
\Vertex[style = labeledStyle, x = 0.700, y = 0.650, L = \footnotesize {$y$}]{v4}
\Vertex[style = unlabeledStyle, x = 0.250, y = 0.570, L = \small {}]{v5}
\Vertex[style = unlabeledStyle, x = 0.350, y = 0.570, L = \small {}]{v6}
\Vertex[style = unlabeledStyle, x = 0.500, y = 0.570, L = \small {}]{v7}
\Vertex[style = unlabeledStyle, x = 0.350, y = 0.880, L = \small {}]{v8}
\Vertex[style = unlabeledStyle, x = 0.650, y = 0.880, L = \small {}]{v9}
\Vertex[style = unlabeledStyle, x = 0.650, y = 0.570, L = \small {}]{v10}
\Vertex[style = unlabeledStyle, x = 0.750, y = 0.570, L = \small {}]{v11}
\Edge[label = \footnotesize {$0$}, labelstyle={xshift=-.08in, fill=none}](v0)(v3)
\Edge[label = \footnotesize {$2$}, labelstyle={xshift=.08in, fill=none}](v2)(v3)
\Edge[label = \footnotesize {$3$}, labelstyle={yshift=.08in, fill=none}](v2)(v0)
\Edge[label = \footnotesize {}, labelstyle={auto=right, fill=none}](v1)(v0)
\Edge[label = \footnotesize {$0$}, labelstyle={xshift=.08in, fill=none}](v4)(v2)
\Edge[label = \footnotesize {$0$}, labelstyle={xshift=-.08in, fill=none}](v5)(v1)
\Edge[label = \footnotesize {$1$}, labelstyle={xshift=.08in, fill=none}](v6)(v1)
\Edge[label = \footnotesize {$1$}, labelstyle={xshift=.08in, fill=none}](v7)(v3)
\Edge[label = \footnotesize {$2$}, labelstyle={xshift=-.08in, fill=none}](v8)(v0)
\Edge[label = \footnotesize {$1$}, labelstyle={xshift=.08in, fill=none}](v9)(v2)
\Edge[label = \footnotesize {$1$}, labelstyle={xshift=-.08in, fill=none}](v10)(v4)
\Edge[label = \footnotesize {$2$}, labelstyle={xshift=.08in, fill=none}](v11)(v4)
\end{tikzpicture}
}}
\subfloat[]{\makebox[.48\textwidth]{
\begin{tikzpicture}[scale = 12]
\tikzstyle{VertexStyle} = []
\tikzstyle{EdgeStyle} = []
\tikzstyle{unlabeledStyle}=[shape = circle, minimum size = 6pt, inner sep = 1.2pt, draw]
\tikzstyle{labeledStyle}=[shape = circle, minimum size = 14pt, inner sep = 1.2pt, draw]
\Vertex[style = labeledStyle, x = 0.400, y = 0.800, L = \footnotesize {$v$}]{v0}
\Vertex[style = labeledStyle, x = 0.300, y = 0.650, L = \footnotesize {$z$}]{v1}
\Vertex[style = labeledStyle, x = 0.600, y = 0.800, L = \footnotesize {$w$}]{v2}
\Vertex[style = labeledStyle, x = 0.500, y = 0.650, L = \footnotesize {$x$}]{v3}
\Vertex[style = labeledStyle, x = 0.700, y = 0.650, L = \footnotesize {$y$}]{v4}
\Vertex[style = unlabeledStyle, x = 0.250, y = 0.570, L = \small {}]{v5}
\Vertex[style = unlabeledStyle, x = 0.350, y = 0.570, L = \small {}]{v6}
\Vertex[style = unlabeledStyle, x = 0.500, y = 0.570, L = \small {}]{v7}
\Vertex[style = unlabeledStyle, x = 0.350, y = 0.880, L = \small {}]{v8}
\Vertex[style = unlabeledStyle, x = 0.650, y = 0.880, L = \small {}]{v9}
\Vertex[style = unlabeledStyle, x = 0.650, y = 0.570, L = \small {}]{v10}
\Vertex[style = unlabeledStyle, x = 0.750, y = 0.570, L = \small {}]{v11}
\Edge[label = \footnotesize {$3$}, labelstyle={xshift=-.08in, fill=none}](v0)(v3)
\Edge[label = \footnotesize {$2$}, labelstyle={xshift=.08in, fill=none}](v2)(v3)
\Edge[label = \footnotesize {$0$}, labelstyle={yshift=.08in, fill=none}](v2)(v0)
\Edge[label = \footnotesize {}, labelstyle={auto=right, fill=none}](v1)(v0)
\Edge[label = \footnotesize {$3$}, labelstyle={xshift=.08in, fill=none}](v4)(v2)
\Edge[label = \footnotesize {$0$}, labelstyle={xshift=-.08in, fill=none}](v5)(v1)
\Edge[label = \footnotesize {$1$}, labelstyle={xshift=.08in, fill=none}](v6)(v1)
\Edge[label = \footnotesize {$1$}, labelstyle={xshift=.08in, fill=none}](v7)(v3)
\Edge[label = \footnotesize {$2$}, labelstyle={xshift=-.08in, fill=none}](v8)(v0)
\Edge[label = \footnotesize {$1$}, labelstyle={xshift=.08in, fill=none}](v9)(v2)
\Edge[label = \footnotesize {$0$}, labelstyle={xshift=-.08in, fill=none}](v10)(v4)
\Edge[label = \footnotesize {$2$}, labelstyle={xshift=.08in, fill=none}](v11)(v4)
\end{tikzpicture}
}}

\caption{Two colorings of $G-vz$ in Case 2.\label{lem1case2}
}
\end{figure}
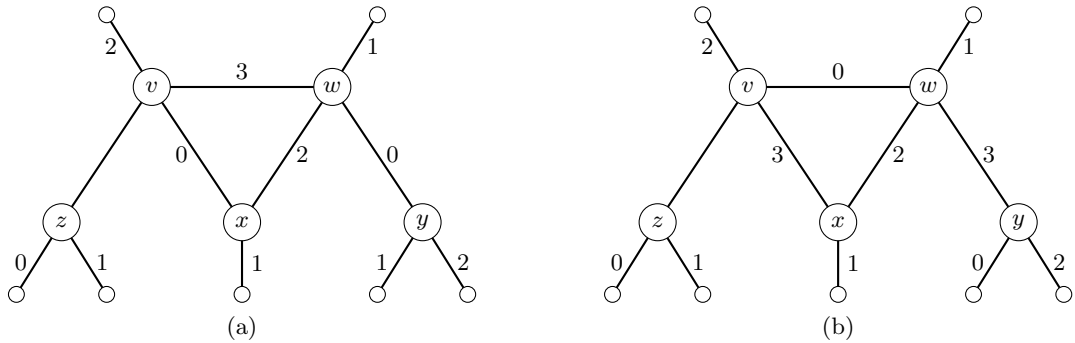

\textbf{Case 2: 0 is used on $\bs{vx}$.} 
The following observation is useful.  If no pendant edge uses 3 and edges $vv'$,
$ww'$, $xx'$ use distinct colors, then $\chi'(G)=4$.
By symmetry, assume that $\vph(vv')=2$, $\vph(ww')=1$, $\vph(xx')=0$, as in
Figure~\ref{lem1case1}(b).  To extend the coloring to $G$, let
$\vph(vz)=3$, $\vph(wy)=3$, $\vph(vw)=0$, $\vph(wx)=2$, $\vph(vx)=1$.

We now show that we may assume the edges are colored as in Figure~\ref{lem1case2}(a).
By symmetry, assume that
$\vph(vv')=2$ and $\vph(vw)=3$.  Note that $\vph(wx)\in\{1,2\}$.  
We assume that $\vph(wx)=2$.  Otherwise
$\vph(wx)=1$ and $\vph(xx')=3$, so a $(1,2)$-swap at $x$ gives $\vph(wx)=2$, as
desired.  Assume $\vph(xx')=3$; otherwise a $(1,3)$-swap at $x$ yields this.
Assume $\vph(ww')=0$ and $\vph(wy)=1$; otherwise a $(1,0)$-swap at $w$ yields
this.  Since $\vph(vw)=3$ and $\vph(wy)=1$, vertex $y$ must see 3.  Also, $y$
must see 2.  If not, then we do a $(1,3)$-swap at $x$, followed by a
$(1,2)$-swap at $y$, and the resulting
$(1,3)$-chain at $v$ ends at $x$.  Now do a $(0,1)$-swap at $y$, followed by
$(1,3)$-swaps at $x$ and $y$ to ensure that $x$ and $y$ each
see 1.  Thus, all edges are colored as in Figure~\ref{lem1case2}(a). 

If a (0,1)-chain in $G-E(H)$ 
starts at either $w$ or $x$ and ends at $\infty$, then we are done by
the observation at the start of Case 2.  So we assume that the $(0,1)$-chain at
$y$ ends at $\infty$, and we recolor it.  To maintain a coloring of $G-vz$, 
let $\vph(vx)=\vph(wy)=3$ and $\vph(vw)=0$, as in Figure~\ref{lem1case2}(b).

Consider the (1,2)-chains in $G-E(H)$ at $v$, $w$, $x$, $y$, $z$.  Let $P$ be
the chain at $w$.  If $P$ ends at $x$, then we recolor it.  To extend the
coloring to $G$, let $\vph(yw)=1$, $\vph(wx)=3$, $\vph(xv)=1$, and
$\vph(vz)=3$.  If $P$ instead ends at $v$, then we again
recolor it; now the extension is the same
as before, except that $\vph(vx)=2$.  So we must consider three possibilities:
the $(1,2)$-chain $P$ at $w$ ends at $y$, ends at $z$, or ends at $\infty$.
In each case, we recolor $P$ and show how to get a coloring of $G$.

Suppose $P$ ends at $y$.  Recolor it, and let $\vph(xw)=0$ and $\vph(wv)=1$, to
maintain a coloring of $G-vz$.  Now consider the $(0,2)$-chain $Q$ at $w$ in
$G-E(H)$.  If $Q$ ends at $y$ or $z$ (or $\infty$), then we are done by the
observation at the start of Case 2.  So assume that $Q$ ends at $v$.  Thus, we
assume the $(0,2)$-chain at $y$ ends at $z$; now we recolor it and let $\vph(vz)=0$.

Suppose $P$ ends at $z$.  Recolor it, and again let $\vph(xw)=0$ and
$\vph(wv)=1$. Now consider the $(0,1)$-chains in $G-E(H)$ that start at $x$,
$y$, and $z$; one of them must end at $\infty$.  If the chain at $z$ ends at
$\infty$, then recolor it and let $\vph(vz)=0$.  If the chain at $x$ ends at
$\infty$, then recolor it, and let $\vph(xw)=1$, $\vph(wv)=0$, and $\vph(vz)=1$.
Finally, if the chain at $y$ ends at $\infty$, then recolor it, and let
$\vph(yw)=0$, $\vph(wx)=3$, $\vph(xv)=0$, and $\vph(vz)=3$.

So assume $P$ ends at $\infty$.  As in the previous case, recolor $P$ and let
$\vph(xw)=0$ and $\vph(wv)=1$.  
Now consider the $(0,2)$-chains in $G-E(H)$ that start at
$v$, $w$, and $z$; one such chain must end at $\infty$, so call it $Q$.  If $Q$
starts at $v$ or $w$, then we recolor it and are done by the observation at the
start of Case 2.  Otherwise $Q$ starts at $z$, so we recolor $Q$ and let
$\vph(vz)=0$.  This completes Case~2.

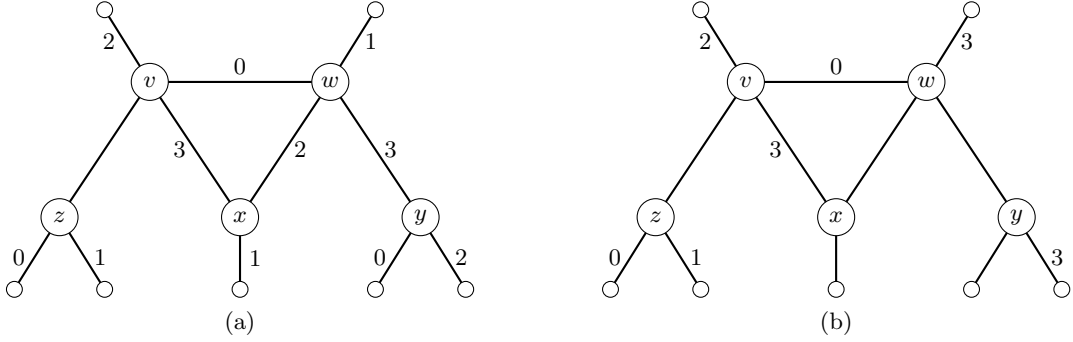
\begin{figure}[!t]
\centering
\subfloat[]{\makebox[.48\textwidth]{
\begin{tikzpicture}[scale = 12]
\tikzstyle{VertexStyle} = []
\tikzstyle{EdgeStyle} = []
\tikzstyle{unlabeledStyle}=[shape = circle, minimum size = 6pt, inner sep = 1.2pt, draw]
\tikzstyle{labeledStyle}=[shape = circle, minimum size = 14pt, inner sep = 1.2pt, draw]
\Vertex[style = labeledStyle, x = 0.400, y = 0.800, L = \footnotesize {$v$}]{v0}
\Vertex[style = labeledStyle, x = 0.300, y = 0.650, L = \footnotesize {$z$}]{v1}
\Vertex[style = labeledStyle, x = 0.600, y = 0.800, L = \footnotesize {$w$}]{v2}
\Vertex[style = labeledStyle, x = 0.500, y = 0.650, L = \footnotesize {$x$}]{v3}
\Vertex[style = labeledStyle, x = 0.700, y = 0.650, L = \footnotesize {$y$}]{v4}
\Vertex[style = unlabeledStyle, x = 0.250, y = 0.570, L = \small {}]{v5}
\Vertex[style = unlabeledStyle, x = 0.350, y = 0.570, L = \small {}]{v6}
\Vertex[style = unlabeledStyle, x = 0.500, y = 0.570, L = \small {}]{v7}
\Vertex[style = unlabeledStyle, x = 0.350, y = 0.880, L = \small {}]{v8}
\Vertex[style = unlabeledStyle, x = 0.650, y = 0.880, L = \small {}]{v9}
\Vertex[style = unlabeledStyle, x = 0.650, y = 0.570, L = \small {}]{v10}
\Vertex[style = unlabeledStyle, x = 0.750, y = 0.570, L = \small {}]{v11}
\Edge[label = \footnotesize {$3$}, labelstyle={xshift=-.08in, fill=none}](v0)(v3)
\Edge[label = \footnotesize {$2$}, labelstyle={xshift=.08in, fill=none}](v2)(v3)
\Edge[label = \footnotesize {$0$}, labelstyle={yshift=.08in, fill=none}](v2)(v0)
\Edge[label = \footnotesize {}, labelstyle={auto=right, fill=none}](v1)(v0)
\Edge[label = \footnotesize {$3$}, labelstyle={xshift=.08in, fill=none}](v4)(v2)
\Edge[label = \footnotesize {$0$}, labelstyle={xshift=-.08in, fill=none}](v5)(v1)
\Edge[label = \footnotesize {$1$}, labelstyle={xshift=.08in, fill=none}](v6)(v1)
\Edge[label = \footnotesize {$1$}, labelstyle={xshift=.08in, fill=none}](v7)(v3)
\Edge[label = \footnotesize {$2$}, labelstyle={xshift=-.08in, fill=none}](v8)(v0)
\Edge[label = \footnotesize {$1$}, labelstyle={xshift=.08in, fill=none}](v9)(v2)
\Edge[label = \footnotesize {$0$}, labelstyle={xshift=-.08in, fill=none}](v10)(v4)
\Edge[label = \footnotesize {$2$}, labelstyle={xshift=.08in, fill=none}](v11)(v4)
\end{tikzpicture}
}}
\subfloat[]{\makebox[.48\textwidth]{
\begin{tikzpicture}[scale = 12]
\tikzstyle{VertexStyle} = []
\tikzstyle{EdgeStyle} = []
\tikzstyle{unlabeledStyle}=[shape = circle, minimum size = 6pt, inner sep = 1.2pt, draw]
\tikzstyle{labeledStyle}=[shape = circle, minimum size = 14pt, inner sep = 1.2pt, draw]
\Vertex[style = labeledStyle, x = 0.400, y = 0.800, L = \footnotesize {$v$}]{v0}
\Vertex[style = labeledStyle, x = 0.300, y = 0.650, L = \footnotesize {$z$}]{v1}
\Vertex[style = labeledStyle, x = 0.600, y = 0.800, L = \footnotesize {$w$}]{v2}
\Vertex[style = labeledStyle, x = 0.500, y = 0.650, L = \footnotesize {$x$}]{v3}
\Vertex[style = labeledStyle, x = 0.700, y = 0.650, L = \footnotesize {$y$}]{v4}
\Vertex[style = unlabeledStyle, x = 0.250, y = 0.570, L = \small {}]{v5}
\Vertex[style = unlabeledStyle, x = 0.350, y = 0.570, L = \small {}]{v6}
\Vertex[style = unlabeledStyle, x = 0.500, y = 0.570, L = \small {}]{v7}
\Vertex[style = unlabeledStyle, x = 0.350, y = 0.880, L = \small {}]{v8}
\Vertex[style = unlabeledStyle, x = 0.650, y = 0.880, L = \small {}]{v9}
\Vertex[style = unlabeledStyle, x = 0.650, y = 0.570, L = \small {}]{v10}
\Vertex[style = unlabeledStyle, x = 0.750, y = 0.570, L = \small {}]{v11}
\Edge[label = \footnotesize {$3$}, labelstyle={xshift=-.08in, fill=none}](v0)(v3)
\Edge[label = \footnotesize {}, labelstyle={xshift=.08in, fill=none}](v2)(v3)
\Edge[label = \footnotesize {$0$}, labelstyle={yshift=.08in, fill=none}](v2)(v0)
\Edge[label = \footnotesize {}, labelstyle={xshift=.08in, fill=none}](v1)(v0)
\Edge[label = \footnotesize {}, labelstyle={xshift=.08in, fill=none}](v4)(v2)
\Edge[label = \footnotesize {$0$}, labelstyle={xshift=-.08in, fill=none}](v5)(v1)
\Edge[label = \footnotesize {$1$}, labelstyle={xshift=.08in, fill=none}](v6)(v1)
\Edge[label = \footnotesize {}, labelstyle={xshift=.08in, fill=none}](v7)(v3)
\Edge[label = \footnotesize {$2$}, labelstyle={xshift=-.08in, fill=none}](v8)(v0)
\Edge[label = \footnotesize {$3$}, labelstyle={xshift=.08in, fill=none}](v9)(v2)
\Edge[label = \footnotesize {}, labelstyle={xshift=.08in, fill=none}](v10)(v4)
\Edge[label = \footnotesize {$3$}, labelstyle={xshift=.08in, fill=none}](v11)(v4)
\end{tikzpicture}
}}

\caption{A coloring and a partial coloring of $G-vz$ in Case 3.\label{lem1case3}
}
\end{figure}

\textbf{Case 3: 0 is used on $\bs{vw}$.} 
By symmetry, assume that $\vph(vv')=2$ and $\vph(vx)=3$. 
Note that $x$ must see both 1 and 2; otherwise a (1,2)-swap at $x$ makes $v$ and
$z$ (1,3)-unlinked, which is a contradiction.  Hence, $x$ misses 0.
Suppose that $\vph(wy)=3$, as in Figure~\ref{lem1case3}(a).  Now $y$ must see 0, or
else a (0,3)-swap at $x$ reduces to Case 2.  We also can assume that $y$ sees 2
and misses 1, since $v$ and $z$ are (1,2)-linked.  Now $\vph(wx)\ne 1$, so
$\vph(wx)=2$, $\vph(ww')=1$, $\vph(xx')=1$.  But now we can (0,1)-swap at one of
$x$ and $y$ without effecting $vw$.  Afterwards, either $x$ misses
1 or $y$ misses 0; in both cases we are done.

So assume that $\vph(ww')=3$, as in Figure~\ref{lem1case3}(b).  Suppose that
$\vph(wx)=2$, so $\vph(wy)=\vph(xx')=1$.  
Now $y$ must see 3 or else a (1,3)-swap at $y$ takes us to
the previous paragraph.  Note that $v$ and $z$ must be (0,2)-linked, or else a
(0,2)-swap at $x$ reduces to Case 1.  Thus, we can assume that $y$ sees 2 and
misses 0.  Similarly, $v$ and $z$ must be (0,3)-linked, or we can reduce to Case 2.
However, now a (0,3)-swap at $y$ makes $y$ miss 3, a contradiction.  So assume
instead that $\vph(wx)=1$ and $\vph(wy)=2$.  Also $\vph(xx')=2$, or else a (1,2)-swap
at $x$ makes $v$ and $z$ (1,3)-unlinked, a contradiction.  
Suppose $y$ misses 0.  Now $y$ and $z$ must be (0,2)-linked, or else a
(0,2)-swap at $y$ reduces to Case 2.  But now a (0,2)-swap at $x$, followed by a
(1,2)-swap at $x$ makes $v$ and $z$ (1,3)-unlinked, a contradiction.  Thus, $y$
sees 0.
Also, $y$ sees exactly one of 1 and 3, and we can assume
it is 3.  But now a (1,2)-swap at $y$ reduces to the case above, where
$\vph(wx)=2$.  
This completes Case 3.
\end{proof}

\begin{lem}
Suppose that $\Delta=4$ and $G$ has the configuration in
Figure~\ref{fig:3reducibles}(b) (reproduced in Figure~\ref{lem2case1}, if we
ignore the colors there). 
If $\chi'(G-ux)=4$, then $\chi'(G)=4$.
\label{lem2}
\end{lem}
\begin{proof}
We start with a coloring of $G-ux$ and assume that $G$ has no
coloring, which leads to a contradiction.  We denote by $u'$, $w'$, and
$x'$ the sole unlabeled neighbors of $u$, $w$, and $x$, respectively.
By symmetry, we assume that $u$ sees 0 and 1, and $x$ sees 0, 2, and 3. 
We repeatedly use that $u$ and $x$ must be (1,2)- and (1,3)-linked.  
We consider two cases: color 0 is used on $uu'$ or used on $uv$.

\textbf{Case 1: 0 is used on $\bs{uu'}$.}
We first show that we may assume the edges are colored as in Figure~\ref{lem2case1}.  
By assumption $\vph(uu')=0$ and $\vph(uv)=1$. 
We show that $w$ must miss 0. If $\vph(vw)=0$, then $\vph(wx)=2$
(by symmetry) and $\vph(ww')=1$, but after a (1,3)-swap at $w$, vertices $u$ and
$x$ are (1,2)-unlinked.  A similar argument works if $\vph(wx)=0$.  If
$\vph(ww')=0$, then $\vph(wx)=2$ (by symmetry), but now $u$ and $x$ are
(1,2)-unlinked, a contradiction.  So $w$ misses 0, as claimed.  So,
by symmetry, we have $\vph(vw)=2$, $\vph(wx)=3$, $\vph(ww')=1$, as in
Figure~\ref{lem2case1}.  

Now we show that $\vph(xx')=0$ and $\vph(xy)=2$.  Suppose to the contrary that
$\vph(xx')=2$ and $\vph(xy)=0$.  
Note that $u$ and $w$ must be (0,2)-linked; otherwise a (0,2)-swap at $w$ gives
$\vph(vw)=0$, a contradiction. 
If we can get $\vph(yz)=2$ and $z$ missing 0,
then a (0,2)-swap at $z$ gives $\vph(xx')=0$ and $\vph(xy)=2$. 
Always $u$ and $x$ must be (1,2)- and (1,3)-linked.  They must also be
(0,3)-linked, since otherwise we get a coloring where $w$ sees 0, a
contradiction.  Now we use a series of (0,2)-, (0,3)-, (1,2)-, and (1,3)-swaps
at $z$ to get $\vph(yz)=2$ and $z$ missing 0.  (If a (0,2)-swap ever recolors
$xx'$ and $xy$, then we accomplish our goal and are done, so we assume this
never happens.) We write $(i;j)$ to denote that
$\vph(yz)=i$ and $z$ misses $j$.  Also $(i;j)\to (i';j')$ if one of the four
swaps mentioned yields $(i';j')$ from $(i;j)$.  We have $(3;0)\to (3;2)\to
(3;1)\to (1;3)\to (1;0)\to (1;2)\to (2;1)\to (2;3)\to (2;0)$. 
So after a (2,0)-swap at $z$, we have $\vph(xx')=0$ and $\vph(xy)=2$, as desired.

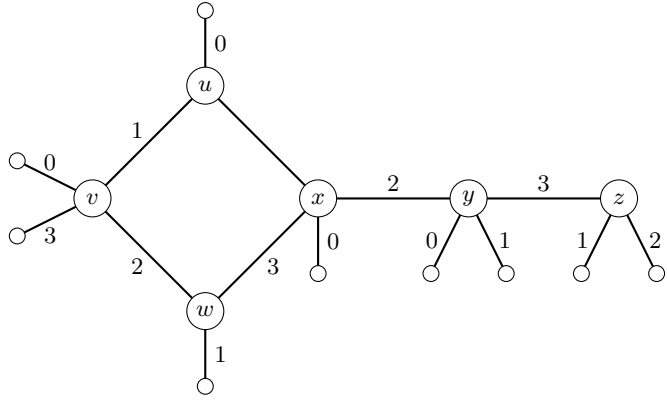
\begin{figure}[!t]
\centering
\begin{tikzpicture}[scale = 10]
\tikzstyle{VertexStyle} = []
\tikzstyle{EdgeStyle} = []
\tikzstyle{labeledStyle}=[shape = circle, minimum size = 14pt, inner sep = 1.2pt, draw]
\tikzstyle{unlabeledStyle}=[shape = circle, minimum size = 6pt, inner sep = 1.2pt, draw]
\Vertex[style = labeledStyle, x = 0.300, y = 0.650, L = \footnotesize {$v$}]{v0}
\Vertex[style = labeledStyle, x = 0.600, y = 0.650, L = \footnotesize {$x$}]{v1}
\Vertex[style = labeledStyle, x = 0.450, y = 0.500, L = \footnotesize {$w$}]{v2}
\Vertex[style = labeledStyle, x = 0.800, y = 0.650, L = \footnotesize {$y$}]{v3}
\Vertex[style = labeledStyle, x = 0.450, y = 0.800, L = \footnotesize {$u$}]{v4}
\Vertex[style = labeledStyle, x = 1.000, y = 0.650, L = \footnotesize {$z$}]{v5}
\Vertex[style = unlabeledStyle, x = 0.450, y = 0.400, L = \small {}]{v6}
\Vertex[style = unlabeledStyle, x = 0.600, y = 0.550, L = \small {}]{v7}
\Vertex[style = unlabeledStyle, x = 1.050, y = 0.550, L = \small {}]{v8}
\Vertex[style = unlabeledStyle, x = 0.950, y = 0.550, L = \small {}]{v9}
\Vertex[style = unlabeledStyle, x = 0.750, y = 0.550, L = \small {}]{v10}
\Vertex[style = unlabeledStyle, x = 0.850, y = 0.550, L = \small {}]{v11}
\Vertex[style = unlabeledStyle, x = 0.200, y = 0.700, L = \small {}]{v12}
\Vertex[style = unlabeledStyle, x = 0.200, y = 0.600, L = \small {}]{v13}
\Vertex[style = unlabeledStyle, x = 0.450, y = 0.900, L = \small {}]{v14}
\Edge[label = \footnotesize {$2$}, labelstyle={xshift=-.06in, yshift=-.06in, fill=none}](v0)(v2)
\Edge[label = \footnotesize {$2$}, labelstyle={yshift=.08in, fill=none}](v3)(v1)
\Edge[label = \footnotesize {$1$}, labelstyle={xshift=-.06in, yshift=.06in, fill=none}](v4)(v0)
\Edge[label = \footnotesize {}, labelstyle={auto=right, fill=none}](v4)(v1)
\Edge[label = \footnotesize {$3$}, labelstyle={yshift=.08in, fill=none}](v5)(v3)
\Edge[label = \footnotesize {$1$}, labelstyle={xshift=.08in, fill=none}](v6)(v2)
\Edge[label = \footnotesize {$0$}, labelstyle={xshift=.08in, fill=none}](v7)(v1)
\Edge[label = \footnotesize {$2$}, labelstyle={xshift=.08in, fill=none}](v8)(v5)
\Edge[label = \footnotesize {$1$}, labelstyle={xshift=.08in, fill=none}](v11)(v3)
\Edge[label = \footnotesize {$3$}, labelstyle={yshift=-.08in, fill=none}](v13)(v0)
\Edge[label = \footnotesize {$0$}, labelstyle={xshift=.08in, fill=none}](v14)(v4)
\Edge[label = \footnotesize {$3$}, labelstyle={xshift=.06in, yshift=-.06in, fill=none}](v2)(v1)
\Edge[label = \footnotesize {$0$}, labelstyle={xshift=-.08in, fill=none}](v3)(v10)
\Edge[label = \footnotesize {$1$}, labelstyle={xshift=-.08in, fill=none}](v5)(v9)
\Edge[label = \footnotesize {$0$}, labelstyle={yshift=.08in, fill=none}](v0)(v12)
\end{tikzpicture}
\caption{A coloring of $G-ux$ in Case 1 of the proof of
Lemma~\ref{lem2}.\label{lem2case1}}
\end{figure}

Finally, we will show that we may assume $\vph(yz)=3$ and $z$ misses 0.
In the notation above, we want to reach the case (3;0).  We can still use
(0,3)-, (1,2)-, and (1,3)-swaps at $z$ (but, in general, cannot use (0,2)-swaps).
We have $(0;2)\to (0;1)\to (0;3)\to (3;0)$.  We also have $(1;0)\to (1;3)\to
(3;1)\to (3;2)$.  Further, in (3;2), we can use a (0,2)-swap to reach (3;0).
For, suppose it interchanges the colors 0 and 2 on $xx'$ and $xy$.  Now the
(0,3)-chain at $z$ ends at $w$.  So a (0,3)-swap at $z$ makes $w$ see 0, a
contradiction.  Finally, consider (1;2).  Now the (1,2)-chain at $z$ ends at
$x$.  After a (1,2)-swap at $z$, we let $\vph(ux)=2$, to get a coloring of
$G$.  So, we may assume the edges are colored as in Figure~\ref{lem2case1}.

Let $H$ be the 6-edge-subgraph induced by $\{u,v,w,x,y,z\}$ of the configuration
in Figure~\ref{lem2case1}.  Consider the (0,1)-chains in $G-E(H)$ at $u$, $v$,
$w$, $x$, $z$.  By parity, one chain must end at $\infty$.
Recall that $w$ and $x$ must be (0,1)-linked in $G$, since $w$ never
sees 0.  So if the (0,1)-chain at $w$ or $x$ ends at $\infty$ or $z$, then we reach a
contradiction.  
If the (0,1)-chain at $v$ ends at $\infty$ or $z$, then recolor it.  Now let
$\vph(vw)=0$, $\vph(uv)=2$, and $\vph(ux)=1$.  
So assume that the (0,1)-chain at $z$ ends at $\infty$, and recolor it. 

Consider the (1,2)-chains in $G-E(H)$ at $w$, $y$, and $z$.
If the (1,2)-chain at $w$ ends at $\infty$ or $z$, then recolor it, and let
$\vph(wv)=1$, $\vph(vu)=2$, and $\vph(ux)=1$.  
So assume the (1,2)-chain at $z$ ends at
$\infty$, and recolor it.  Finally, consider the (1,3)-chains in $G-E(H)$ that
start at $v$, $w$, $y$, and $z$.  If the chain at $z$ ends at $y$, then recolor
it, and let $\vph(zy)=2$, $\vph(yx)=1$, and $\vph(xu)=2$.  If the chain at $z$
ends at $w$, then recolor it and let $\vph(zy)=2$, $\vph(yx)=3$, $\vph(xw)=1$,
and $\vph(xu)=2$.  If the chain at $y$ ends at $w$, then recolor it and
let $\vph(zy)=2$, $\vph(yx)=1$, $\vph(xw)=2$, $\vph(wv)=1$,
$\vph(vu)=2$, and $\vph(ux)=3$.  This finishes Case 1.

\textbf{Case 2: 0 is used on $\bs{uv}$.}
We show that we may assume all edges are colored as in Figure~\ref{lem2case2}. 
After that, a (0,2)-swap at $w$ gives $\vph(wx)=0$, which we handle now.
Suppose first that $\vph(wx)=0$.  By possibly using a (1,2)- or (1,3)-swap
at  $w$, we assume that $w$ misses 1.  Now we let $\vph(wx)=1$, which reduces
to Case 1.  So we assume, by symmetry, that $\vph(wx)=2$.  If $w$ sees 0, then
(possibly after a (1,3)-swap at $w$), vertex $w$ misses 1, so $u$ and $x$ are
(1,2)-unlinked, a contradiction.  Thus, $w$ misses 0.  Suppose that $\vph(vw)=1$ and
$\vph(ww')=3$.  Now we uncolor $wx$ and let $\vph(ux)=2$.  This reduces to Case
1, with $w$ in place of $u$ (and 3 in place of 0).  So we assume that
$\vph(vw)=3$ and $\vph(ww')=1$, as in Figure~\ref{lem2case2}. Assume that
$\vph(xx')=3$ and $\vph(xy)=0$; otherwise, this follows from a (0,3)-swap at $x$.

As in Case 1, we write $(i;j)$ to denote that $\vph(yz)=i$ and $z$ misses $j$.
Recall that (1,2)- and (1,3)-swaps at $z$ do not change the colors on edges
incident to $u$, $v$, $w$, and $x$.  Neither do (0,2)-swaps, when $\vph(yz)\ne
2$. (If $w$ and $u$ are (0,2)-unlinked, then we can get $\vph(wx)=0$, which
reduces to Case 1, as in the previous paragraph.)  In fact, we can also use (0,1)-swaps,
as follows.  Vertices $u$ and $x$ must be (0,1)-linked, or we reduce to Case 1.
And $w$ and $x$ must be (0,1)-linked, or else we get $w$ missing 1, which is a
contradiction.

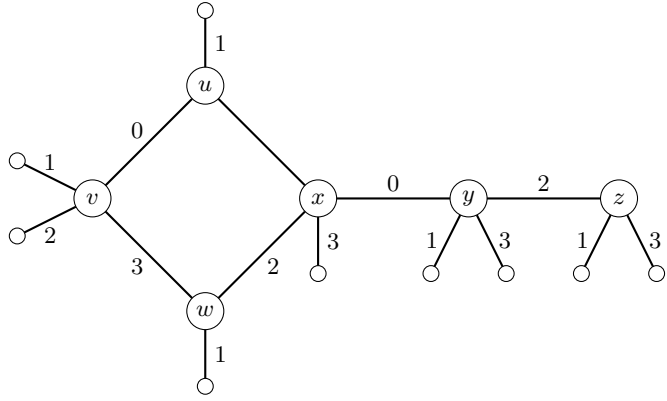
\begin{figure}[!t]
\centering
\begin{tikzpicture}[scale = 10]
\tikzstyle{VertexStyle} = []
\tikzstyle{EdgeStyle} = []
\tikzstyle{labeledStyle}=[shape = circle, minimum size = 14pt, inner sep = 1.2pt, draw]
\tikzstyle{unlabeledStyle}=[shape = circle, minimum size = 6pt, inner sep = 1.2pt, draw]
\Vertex[style = labeledStyle, x = 0.300, y = 0.650, L = \footnotesize {$v$}]{v0}
\Vertex[style = labeledStyle, x = 0.600, y = 0.650, L = \footnotesize {$x$}]{v1}
\Vertex[style = labeledStyle, x = 0.450, y = 0.500, L = \footnotesize {$w$}]{v2}
\Vertex[style = labeledStyle, x = 0.800, y = 0.650, L = \footnotesize {$y$}]{v3}
\Vertex[style = labeledStyle, x = 0.450, y = 0.800, L = \footnotesize {$u$}]{v4}
\Vertex[style = labeledStyle, x = 1.000, y = 0.650, L = \footnotesize {$z$}]{v5}
\Vertex[style = unlabeledStyle, x = 0.450, y = 0.400, L = \small {}]{v6}
\Vertex[style = unlabeledStyle, x = 0.600, y = 0.550, L = \small {}]{v7}
\Vertex[style = unlabeledStyle, x = 1.050, y = 0.550, L = \small {}]{v8}
\Vertex[style = unlabeledStyle, x = 0.950, y = 0.550, L = \small {}]{v9}
\Vertex[style = unlabeledStyle, x = 0.750, y = 0.550, L = \small {}]{v10}
\Vertex[style = unlabeledStyle, x = 0.850, y = 0.550, L = \small {}]{v11}
\Vertex[style = unlabeledStyle, x = 0.200, y = 0.700, L = \small {}]{v12}
\Vertex[style = unlabeledStyle, x = 0.200, y = 0.600, L = \small {}]{v13}
\Vertex[style = unlabeledStyle, x = 0.450, y = 0.900, L = \small {}]{v14}
\Edge[label = \footnotesize {$3$}, labelstyle={xshift=-.06in, yshift=-.06in, fill=none}](v0)(v2)
\Edge[label = \footnotesize {$0$}, labelstyle={yshift=.08in, fill=none}](v3)(v1)
\Edge[label = \footnotesize {$0$}, labelstyle={xshift=-.06in, yshift=.06in, fill=none}](v4)(v0)
\Edge[label = \footnotesize {}, labelstyle={auto=right, fill=none}](v4)(v1)
\Edge[label = \footnotesize {$2$}, labelstyle={yshift=.08in, fill=none}](v5)(v3)
\Edge[label = \footnotesize {$1$}, labelstyle={xshift=.08in, fill=none}](v6)(v2)
\Edge[label = \footnotesize {$3$}, labelstyle={xshift=.08in, fill=none}](v7)(v1)
\Edge[label = \footnotesize {$3$}, labelstyle={xshift=.08in, fill=none}](v8)(v5)
\Edge[label = \footnotesize {$3$}, labelstyle={xshift=.08in, fill=none}](v11)(v3)
\Edge[label = \footnotesize {$2$}, labelstyle={yshift=-.08in, fill=none}](v13)(v0)
\Edge[label = \footnotesize {$1$}, labelstyle={xshift=.08in, fill=none}](v14)(v4)
\Edge[label = \footnotesize {$2$}, labelstyle={xshift=.06in, yshift=-.06in, fill=none}](v2)(v1)
\Edge[label = \footnotesize {$1$}, labelstyle={xshift=-.08in, fill=none}](v3)(v10)
\Edge[label = \footnotesize {$1$}, labelstyle={xshift=-.08in, fill=none}](v5)(v9)
\Edge[label = \footnotesize {$1$}, labelstyle={yshift=.08in, fill=none}](v0)(v12)

\end{tikzpicture}
\caption{A coloring of $G-ux$ in Case 2 of the proof of
Lemma~\ref{lem2}.\label{lem2case2}}
\end{figure}

We write $(i;j)\to (i';j')$ if, starting from $(i;j)$, we get $(i';j')$ by using
a (0,1)-, (0,2)-, (1,2)-, or (1,3)-swap at $z$.  Our goal is to reach (2;0), as
in Figure~\ref{lem2case2}.  When we do, the (0,2)-chain at $z$ ends at $w$.  Now a
(0,2)-swap at $z$ gives $\vph(wx)=0$, which we handled in the first paragraph.
Note that $(1;0)\to (1;2)\to (2;1)\to (2;0)$.  Also, $(2;3)\to (2;1)\to (2;0)$.
In any of these five cases, we are done.
Note also that $(3;2)\to (3;0)\to (3;1)\to (1;3)$, so we can assume (1;3).
Now we use (0,3)-swaps at $x$ and $z$.  
So we have (1;0) and $\vph(xx')=0$ and $\vph(xy)=3$.  We use a (0,2)-swap at
$z$, followed by a (0,3)-swap at $x$, followed by a (1,2)-swap at $z$, followed
by a (0,1)-swap at $z$.
Now all edges are colored as in Figure~\ref{lem2case2}, so we are done.
\end{proof}

\begin{lem}
Suppose that $G\in \HH_4$ and $G$ has the configuration in
Figure~\ref{fig:3reducibles}(c) (reproduced in Figure~\ref{lem3case1}, if we
ignore the colors there). 
If $\chi'(G-st)=4$, then $\chi'(G)=4$.
\label{lem3}
\end{lem}
\begin{proof}
We start with a coloring of $G-st$ and assume that $G$ has no
coloring, which leads to a contradiction.  
We denote by $v'$ the unlabeled neighbor of $v$.
By symmetry, we assume that
$t$ sees 0 and 1, and $s$ sees 0, 2, and 3.  
We consider two cases: either 0 is used on $su$ or it is not.


\begin{figure}[!t]
\centering
\begin{tikzpicture}[scale = 10]
\tikzstyle{VertexStyle} = []
\tikzstyle{EdgeStyle} = []
\tikzstyle{labeledStyle}=[shape = circle, minimum size = 14pt, inner sep = 1.2pt, draw]
\tikzstyle{unlabeledStyle}=[shape = circle, minimum size = 6pt, inner sep = 1.2pt, draw]
\Vertex[style = labeledStyle, x = 0.600, y = 0.850, L = \footnotesize {$s$}]{v0}
\Vertex[style = labeledStyle, x = 0.300, y = 0.650, L = \footnotesize {$t$}]{v1}
\Vertex[style = labeledStyle, x = 0.500, y = 0.650, L = \footnotesize {$u$}]{v2}
\Vertex[style = labeledStyle, x = 0.700, y = 0.650, L = \footnotesize {$v$}]{v3}
\Vertex[style = labeledStyle, x = 0.900, y = 0.650, L = \footnotesize {$w$}]{v4}
\Vertex[style = labeledStyle, x = 0.600, y = 0.450, L = \footnotesize {$x$}]{v5}
\Vertex[style = labeledStyle, x = 0.800, y = 0.450, L = \footnotesize {$y$}]{v6}
\Vertex[style = labeledStyle, x = 1.000, y = 0.450, L = \footnotesize {$z$}]{v7}
\Vertex[style = unlabeledStyle, x = 0.250, y = 0.550, L = \tiny {}]{v8}
\Vertex[style = unlabeledStyle, x = 0.350, y = 0.550, L = \tiny {}]{v9}
\Vertex[style = unlabeledStyle, x = 0.450, y = 0.550, L = \tiny {}]{v10}
\Vertex[style = unlabeledStyle, x = 0.550, y = 0.550, L = \tiny {}]{v11}
\Vertex[style = unlabeledStyle, x = 0.600, y = 0.650, L = \tiny {}]{v12}
\Vertex[style = unlabeledStyle, x = 1.050, y = 0.700, L = \tiny {}]{v13}
\Vertex[style = unlabeledStyle, x = 1.050, y = 0.650, L = \tiny {}]{v14}
\Vertex[style = unlabeledStyle, x = 0.550, y = 0.350, L = \tiny {}]{v15}
\Vertex[style = unlabeledStyle, x = 0.650, y = 0.350, L = \tiny {}]{v16}
\Vertex[style = unlabeledStyle, x = 0.750, y = 0.350, L = \tiny {}]{v17}
\Vertex[style = unlabeledStyle, x = 0.850, y = 0.350, L = \tiny {}]{v18}
\Vertex[style = unlabeledStyle, x = 0.950, y = 0.350, L = \tiny {}]{v19}
\Vertex[style = unlabeledStyle, x = 1.050, y = 0.350, L = \tiny {}]{v20}
\Edge[label = \footnotesize {$2$}, labelstyle={xshift=.08in, fill=none}](v3)(v0)
\Edge[label = \footnotesize {$3$}, labelstyle={xshift=.15in, fill=none}](v4)(v0)
\Edge[label = \footnotesize {$0$}, labelstyle={xshift=-.08in, fill=none}](v1)(v8)
\Edge[label = \footnotesize {$1$}, labelstyle={xshift=.08in, fill=none}](v9)(v1)
\Edge[label = \footnotesize {$3$}, labelstyle={xshift=.08in, fill=none}](v11)(v2)
\Edge[label = \footnotesize {}, labelstyle={xshift=.08in, fill=none}](v20)(v7)
\Edge[label = \footnotesize {$1$}, labelstyle={xshift=-.08in, fill=none}](v2)(v10)
\Edge[label = \footnotesize {$0$}, labelstyle={xshift=-.08in, fill=none}](v0)(v2)
\Edge[label = \footnotesize {}, labelstyle={xshift=-.08in, fill=none}](v7)(v19)
\Edge[label = \footnotesize {}, labelstyle={xshift=.08in, fill=none}](v7)(v4)
\Edge[label = \footnotesize {}, labelstyle={auto=right, fill=none}](v1)(v0)
\Edge[label = \footnotesize {$1$}, labelstyle={xshift=.08in, fill=none}](v6)(v3)
\Edge[label = \footnotesize {}, labelstyle={xshift=.08in, fill=none}](v13)(v4)
\Edge[label = \footnotesize {$2$}, labelstyle={xshift=.08in, fill=none}](v16)(v5)
\Edge[label = \footnotesize {$2$}, labelstyle={xshift=.08in, fill=none}](v18)(v6)
\Edge[label = \footnotesize {$0$}, labelstyle={xshift=-.08in, fill=none}](v3)(v5)
\Edge[label = \footnotesize {$1$}, labelstyle={xshift=-.08in, fill=none}](v5)(v15)
\Edge[label = \footnotesize {$0$}, labelstyle={xshift=-.08in, fill=none}](v6)(v17)
\Edge[label = \footnotesize {}, labelstyle={auto=right, fill=none}](v4)(v14)
\Edge[label = \footnotesize {$3$}, labelstyle={yshift=.08in, fill=none}](v3)(v12)
\end{tikzpicture}
\caption{A partial coloring of $G-st$ in Case 1 of the proof of
Lemma~\ref{lem3}.\label{lem3case1}}
\end{figure}
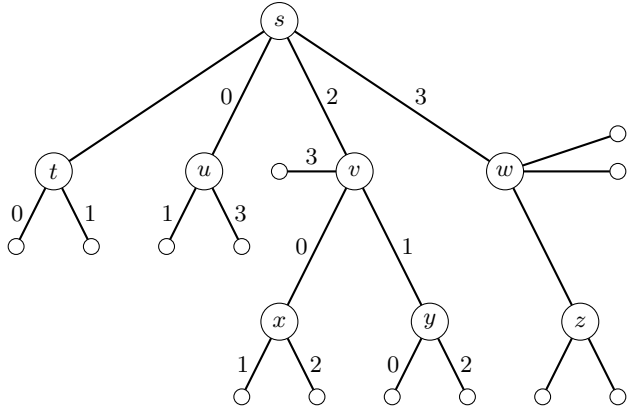

\textbf{Case 1: 0 is used on $\bs{su}$.}  By symmetry, assume that $\vph(su)=0$,
$\vph(sv)=2$, $\vph(sw)=3$.  We will either reach the coloring in
Figure~\ref{lem3case1} (which we show how to extend to $st$ at the end of this
case) or else reduce to Case 2: $\vph(su)\ne 0$.
Since $s$ and $t$ are (1,2)- and (1,3)-linked, we use (1,2)- and
(1,3)-swaps at $u$ to get $u$ missing 2 (without changing colors on edges
incident to $s$ and $t$).

We show that we may assume $\vph(vx)=0$.  Suppose not; by symmetry between $x$ and $y$,
assume that $\vph(vx)=3$, $\vph(vy)=1$, and $\vph(vv')=0$.  
Now $y$ sees 2, since $s$ and $t$ are (1,2)-linked. 
And $u$ must be (0,2)-linked to $t$ (possibly through $w$ and $z$) or else we
(0,2)-swap at $u$ and reduce to Case 2.  If $y$ misses 0, then a (0,2)-swap at
$y$ makes $y$ miss 2 (and thus $s$ and $t$ are (1,2)-unlinked).  So assume
$y$ sees 0.  After a (1,3)-swap at $y$, we have $\vph(vy)=3$ and $\vph(vx)=1$,
with nothing else changed.  So, by the argument above, $x$ sees 0 and 2.  But
now the (1,3)-chain at $x$ ends at $y$.  So either $s$ and $t$ are currently
(1,2)-unlinked, or else they become so after a (1,3)-swap at $x$.  Thus, we
conclude that $\vph(vx)=0$.

Now we show that we may assume $\vph(vy)=1$.  Assume instead that $\vph(vy)=3$ and 
$\vph(vv')=1$.  
Now $x$ sees 2, or else a (0,2)-swap at $x$ reduces to Case 2.
If necessary, use a (1,3)-swap at $x$ to get $x$ missing 3.  
Note that $y$ sees 1, or else a (1,3)-swap at $y$ gives $\vph(vy)=1$.  
If necessary, use a (0,2)-swap at $y$ to get $y$ missing 0.
Now the (0,3)-chain at $x$ ends at $y$.  So either $u$ and $t$ are
(0,2)-unlinked, or else they become so after a (0,3)-swap at $x$.  In either
case, we use a (0,2)-swap at $u$ to reduce to Case 2.  So we must have
$\vph(vx)=0$, $\vph(vy)=1$, $\vph(vv')=3$.

Since $s$ and $t$ are (0,2)- and (1,2)-linked, both $x$ and $y$ see 2.
Further, if $y$ misses 0, then after a (0,2)-swap at $y$ vertices
$s$ and $t$ are (1,2)-unlinked.  Thus, $y$ sees 0 and 2, and misses 3.
Suppose $x$ misses 1.  Now the (1,2)-chain $P$ at $x$ must end at $u$ (possibly via $w$
and $z$); otherwise we recolor $P$, which makes $u$ and $t$ (0,2)-unlinked, a
contradiction.  Now recolor $P$ and let $\vph(us)=1$, $\vph(sv)=0$, $\vph(vx)=2$,
$\vph(st)=2$.  
%
So instead $x$ sees 1 and misses 3, as in Figure~\ref{lem3case1}.  Now recolor the
(1,3)-chains at $x$ and $y$ (possibly the same chain).  Again, the (1,2)-chain
$Q$ at $x$ must end at $u$ (possibly via $w$ and $z$), since otherwise we
recolor it and $u$ and $t$ are (0,2)-unlinked.  Now Recolor $Q$;  as before,
let $\vph(us)=1$, $\vph(sv)=0$, $\vph(vx)=2$, $\vph(st)=2$.
This completes Case 1.

\textbf{Case 2: 0 is not used on $\bs{su}$.}  
By assumption $s$ sees 0.  We show that we may assume $\vph(sw)=0$.  Suppose instead
that $\vph(sv)=0$.  Since $G\in \HH_4$, vertex $w$ has two 3-neighbors.  
If $u\in N(w)$ or $t\in N(w)$, then we have an instance of
Figure~\ref{fig:3reducibles}(a), since $z\notin\{t,u\}$. 
So $\chi'(G)=4$, by Lemma~\ref{lem1}.
Thus, we assume $t,u\notin N(w)$.
Now we interchange the roles of $v$ and $w$. (Vertices $v$ and $w$ could have
a common 3-neighbor, but this is not a problem.) So $\vph(sw)=0$.  By symmetry,
assume that $\vph(su)=2$ and $\vph(sv)=3$.  Since $s$ and $t$ are
(1,2)-linked, $u$ must see 1.  If $u$ misses 3, then after a (1,3)-swap at
$u$, vertices $s$ and $t$ are (1,2)-unlinked.  So $u$ must miss 0.  Thus, we
have Figure~\ref{lem3case2}(a), except for colors on edges incident to $z$.

\begin{figure}[!t]
\centering
\subfloat[]{\makebox[.48\textwidth]{
\begin{tikzpicture}[scale = 9]
\tikzstyle{VertexStyle} = []
\tikzstyle{EdgeStyle} = []
\tikzstyle{labeledStyle}=[shape = circle, minimum size = 14pt, inner sep = 1.2pt, draw]
\tikzstyle{unlabeledStyle}=[shape = circle, minimum size = 6pt, inner sep = 1.2pt, draw]
\Vertex[style = labeledStyle, x = 0.600, y = 0.850, L = \footnotesize {$s$}]{v0}
\Vertex[style = labeledStyle, x = 0.300, y = 0.650, L = \footnotesize {$t$}]{v1}
\Vertex[style = labeledStyle, x = 0.500, y = 0.650, L = \footnotesize {$u$}]{v2}
\Vertex[style = labeledStyle, x = 0.700, y = 0.650, L = \footnotesize {$v$}]{v3}
\Vertex[style = labeledStyle, x = 0.900, y = 0.650, L = \footnotesize {$w$}]{v4}
\Vertex[style = labeledStyle, x = 0.600, y = 0.450, L = \footnotesize {$x$}]{v5}
\Vertex[style = labeledStyle, x = 0.800, y = 0.450, L = \footnotesize {$y$}]{v6}
\Vertex[style = labeledStyle, x = 1.000, y = 0.450, L = \footnotesize {$z$}]{v7}
\Vertex[style = unlabeledStyle, x = 0.250, y = 0.550, L = \tiny {}]{v8}
\Vertex[style = unlabeledStyle, x = 0.350, y = 0.550, L = \tiny {}]{v9}
\Vertex[style = unlabeledStyle, x = 0.450, y = 0.550, L = \tiny {}]{v10}
\Vertex[style = unlabeledStyle, x = 0.550, y = 0.550, L = \tiny {}]{v11}
\Vertex[style = unlabeledStyle, x = 0.600, y = 0.650, L = \tiny {}]{v12}
\Vertex[style = unlabeledStyle, x = 1.050, y = 0.700, L = \tiny {}]{v13}
\Vertex[style = unlabeledStyle, x = 1.050, y = 0.650, L = \tiny {}]{v14}
\Vertex[style = unlabeledStyle, x = 0.550, y = 0.350, L = \tiny {}]{v15}
\Vertex[style = unlabeledStyle, x = 0.650, y = 0.350, L = \tiny {}]{v16}
\Vertex[style = unlabeledStyle, x = 0.750, y = 0.350, L = \tiny {}]{v17}
\Vertex[style = unlabeledStyle, x = 0.850, y = 0.350, L = \tiny {}]{v18}
\Vertex[style = unlabeledStyle, x = 0.950, y = 0.350, L = \tiny {}]{v19}
\Vertex[style = unlabeledStyle, x = 1.050, y = 0.350, L = \tiny {}]{v20}
\Edge[label = \tiny {}, labelstyle={auto=right, fill=none}](v1)(v0)
\Edge[label = \footnotesize {$3$}, labelstyle={xshift=.08in, fill=none}](v3)(v0)
\Edge[label = \footnotesize {$0$}, labelstyle={xshift=.15in, fill=none}](v4)(v0)
\Edge[label = \footnotesize {$0$}, labelstyle={xshift=-.08in, fill=none}](v1)(v8)
\Edge[label = \footnotesize {$1$}, labelstyle={xshift=.08in, fill=none}](v9)(v1)
\Edge[label = \footnotesize {$3$}, labelstyle={xshift=.08in, fill=none}](v11)(v2)
\Edge[label = \footnotesize {$2$}, labelstyle={xshift=.08in, fill=none}](v20)(v7)
\Edge[label = \footnotesize {$1$}, labelstyle={xshift=-.08in, fill=none}](v2)(v10)
\Edge[label = \footnotesize {$2$}, labelstyle={xshift=-.08in, fill=none}](v0)(v2)
\Edge[label = \footnotesize {$1$}, labelstyle={xshift=-.08in, fill=none}](v7)(v19)
\Edge[label = \footnotesize {$3$}, labelstyle={xshift=.08in, fill=none}](v7)(v4)
\Edge[label = \small {}, labelstyle={auto=right, fill=none}](v6)(v3)
\Edge[label = \small {}, labelstyle={auto=right, fill=none}](v13)(v4)
\Edge[label = \small {}, labelstyle={auto=right, fill=none}](v16)(v5)
\Edge[label = \small {}, labelstyle={auto=right, fill=none}](v18)(v6)
\Edge[label = \small {}, labelstyle={auto=right, fill=none}](v3)(v5)
\Edge[label = \small {}, labelstyle={auto=right, fill=none}](v5)(v15)
\Edge[label = \small {}, labelstyle={auto=right, fill=none}](v6)(v17)
\Edge[label = \small {}, labelstyle={auto=right, fill=none}](v4)(v14)
\Edge[label = \small {}, labelstyle={auto=right, fill=none}](v3)(v12)
\end{tikzpicture}
}}
\subfloat[]{\makebox[.48\textwidth]{
\begin{tikzpicture}[scale = 9]
\tikzstyle{VertexStyle} = []
\tikzstyle{EdgeStyle} = []
\tikzstyle{unlabeledStyle}=[shape = circle, minimum size = 6pt, inner sep = 1.2pt, draw]
\tikzstyle{labeledStyle}=[shape = circle, minimum size = 14pt, inner sep = 1.2pt, draw]
%
\Vertex[style = labeledStyle, x = 0.600, y = 0.850, L = \footnotesize {$s$}]{v0}
\Vertex[style = labeledStyle, x = 0.300, y = 0.650, L = \footnotesize {$t$}]{v1}
\Vertex[style = labeledStyle, x = 0.500, y = 0.650, L = \footnotesize {$u$}]{v2}
\Vertex[style = labeledStyle, x = 0.700, y = 0.650, L = \footnotesize {$v$}]{v3}
\Vertex[style = labeledStyle, x = 0.900, y = 0.650, L = \footnotesize {$w$}]{v4}
\Vertex[style = labeledStyle, x = 0.600, y = 0.450, L = \footnotesize {$x$}]{v5}
\Vertex[style = labeledStyle, x = 0.800, y = 0.450, L = \footnotesize {$y$}]{v6}
\Vertex[style = labeledStyle, x = 1.000, y = 0.450, L = \footnotesize {$z$}]{v7}
\Vertex[style = unlabeledStyle, x = 0.250, y = 0.550, L = \tiny {}]{v8}
\Vertex[style = unlabeledStyle, x = 0.350, y = 0.550, L = \tiny {}]{v9}
\Vertex[style = unlabeledStyle, x = 0.450, y = 0.550, L = \tiny {}]{v10}
\Vertex[style = unlabeledStyle, x = 0.550, y = 0.550, L = \tiny {}]{v11}
\Vertex[style = unlabeledStyle, x = 0.600, y = 0.650, L = \tiny {}]{v12}
\Vertex[style = unlabeledStyle, x = 1.050, y = 0.650, L = \tiny {}]{v14}
\Vertex[style = unlabeledStyle, x = 0.550, y = 0.350, L = \tiny {}]{v15}
\Vertex[style = unlabeledStyle, x = 0.650, y = 0.350, L = \tiny {}]{v16}
\Vertex[style = unlabeledStyle, x = 0.800, y = 0.350, L = \tiny {}]{v18}
\Vertex[style = unlabeledStyle, x = 0.950, y = 0.350, L = \tiny {}]{v19}
\Vertex[style = unlabeledStyle, x = 1.050, y = 0.350, L = \tiny {}]{v20}
\Edge[label = \footnotesize {$3$}, labelstyle={xshift=.08in, fill=none}](v3)(v0)
\Edge[label = \footnotesize {$0$}, labelstyle={xshift=.15in, fill=none}](v4)(v0)
\Edge[label = \footnotesize {$0$}, labelstyle={xshift=-.08in, fill=none}](v1)(v8)
\Edge[label = \footnotesize {$1$}, labelstyle={xshift=.08in, fill=none}](v9)(v1)
\Edge[label = \footnotesize {$3$}, labelstyle={xshift=.08in, fill=none}](v11)(v2)
\Edge[label = \footnotesize {$3$}, labelstyle={xshift=.08in, fill=none}](v20)(v7)
\Edge[label = \footnotesize {$1$}, labelstyle={xshift=-.08in, fill=none}](v2)(v10)
\Edge[label = \footnotesize {$2$}, labelstyle={xshift=-.08in, fill=none}](v0)(v2)
\Edge[label = \footnotesize {$1$}, labelstyle={xshift=-.08in, fill=none}](v7)(v19)
\Edge[label = \footnotesize {$2$}, labelstyle={xshift=.08in, fill=none}](v7)(v4)
\Edge[label = \footnotesize {$3$}, labelstyle={xshift=.08in, fill=none}](v6)(v4)
\Edge[label = \footnotesize {$1$}, labelstyle={yshift=.08in, fill=none}](v4)(v14)
\Edge[label = \small {}, labelstyle={auto=right, fill=none}](v1)(v0)
\Edge[label = \small {}, labelstyle={auto=right, fill=none}](v6)(v3)
\Edge[label = \small {}, labelstyle={auto=right, fill=none}](v16)(v5)
\Edge[label = \small {}, labelstyle={auto=right, fill=none}](v18)(v6)
\Edge[label = \small {}, labelstyle={auto=right, fill=none}](v3)(v5)
\Edge[label = \small {}, labelstyle={auto=right, fill=none}](v5)(v15)
\Edge[label = \small {}, labelstyle={auto=right, fill=none}](v3)(v12)
\end{tikzpicture}
}}
\captionsetup{
  margin = 74pt
}
\caption{Two partial colorings of $G-st$ in Case 2 of Lemma~\ref{lem3}.  (a) A
partial coloring of $G-st$. (b) A partial coloring of $G-st$, when $G$ also has
edge $wy$.\label{lem3case2}}
\end{figure}
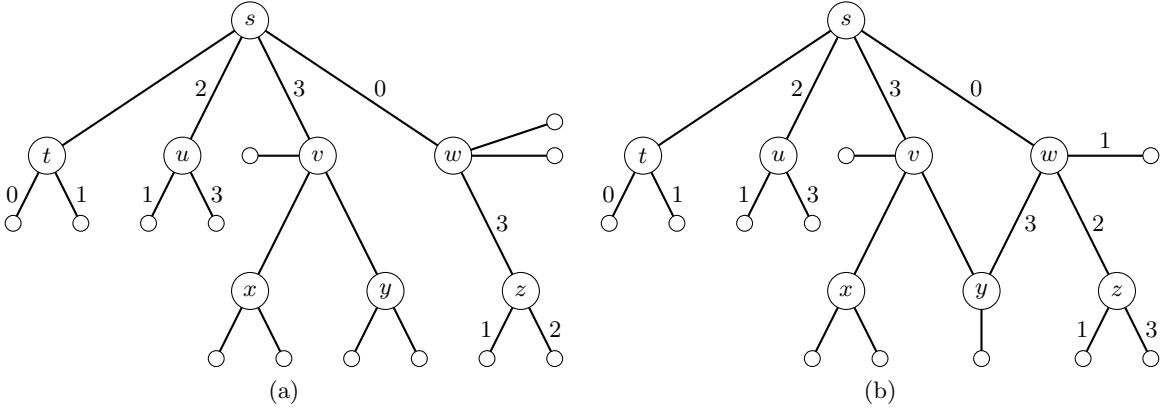

We show that we may assume we have either Figure~\ref{lem3case2}(a) or else
Figure~\ref{lem3case2}(b) with $y$ missing 0.  Note that $u$ and $s$ are 
(0,1)-linked, since otherwise we (0,1)-swap at $u$ and finish as above.
First, we get $z$ missing 0.  If $z$ sees 0 and misses 1, then we (0,1)-swap
at $z$.  Otherwise, if $z$ sees 0 it misses 2 or 3, so after a (1,2)- or
(1,3)-swap, $z$ misses 1.  These swaps at $z$ do not change the colors
on edges incident to $s$, $t$, or $u$, since $s$ and $t$ are (1,2)- and
(1,3)-linked and $s$ and $u$ are (0,1)-linked.  Thus $z$ misses 0.
Now consider $\vph(wz)$.
If $\vph(wz)=1$, then the (0,1)-chain at $z$ ends at $s$, so we recolor it and
let $\vph(su)=0$ and $\vph(st)=2$.  If $\vph(wz)=3$, then we are in
Figure~\ref{lem3case2}(a).  So assume $\vph(wz)=2$.

Now consider the 3-neighbor $\hat{z}$ of $w$, other than $z$.  As in the first
paragraph of Case 2, we know $\hat{z}\notin\{t,u\}$.  First suppose
$\hat{z}\notin\{x,y\}$.  Now we essentially repeat the argument above, with
$\hat{z}$ in place of $z$.  Suppose $\vph(w\hat{z})=3$.  If $\hat{z}$ misses 0,
then we have Figure~\ref{lem3case2}(a), with $\hat{z}$ in place of $z$.  If $\hat{z}$
misses 1, then a (0,1)-swap at $\hat{z}$ gives that $\hat{z}$ misses 0, and we
again reach Figure~\ref{lem3case2}(a); this could make $z$ miss 1, but that is
irrelevant. If $\hat{z}$ misses 2, then a (1,2)-swap at $\hat{z}$ gives that
$\hat{z}$ misses 1.  
So instead assume $\vph(w\hat{z})=1$.  If $\hat{z}$ misses 0, then we use a
(0,1)-swap at $\hat{z}$, as above.  If $\hat{z}$ misses 2, then a (1,2)-swap at
$\hat{z}$ makes $\vph(wz)=1$ and $z$ still misses 0, so we are done.  So assume
$\hat{z}$ misses 3.  Now a (1,3)-swap at $\hat{z}$ reduces to the case above where
$\vph(w\hat{z})=3$.
This concludes the case where $\hat{z}\notin\{x,y\}$.  
Now suppose $\hat{z}\in\{x,y\}$; by symmetry, assume that $\hat{z}=y$.
This case is identical, except that we end in Figure~\ref{lem3case2}(b) with $y$
missing 0.  Thus, we may assume we have either Figure~\ref{lem3case2}(a) or
Figure~\ref{lem3case2}(b) with $y$ missing 0.
We first consider the latter case, since the argument is simpler.

\textbf{Case 2a: we have Figure~\ref{lem3case2}(b) with $\bs{y}$ missing 0.}
Clearly $\vph(vx)$ is 2, 1, or 0.
First suppose that $\vph(vx)=2$, which implies $\vph(vy)=1$.
Now let $\vph(vy)=3$, $\vph(yw)=0$, $\vph(ws)=3$, $\vph(sv)=1$, $\vph(su)=0$,
$\vph(st)=2$.  So $\vph(vx)\ne 2$. In what follows, we often use variations on
this recoloring idea, typically letting $\vph(vy)=\vph(ws)=3$ and $\vph(wy)=0$,
and also recoloring some other edges.

Suppose instead that $\vph(vx)=1$, which implies $\vph(vy)=2$; recall that $y$
misses 0.  Now $x$ must see 3, so $x$ misses 2 or 0.  If $x$ misses 2, then
let $\vph(vx)=2$, $\vph(vy)=3$, $\vph(yw)=0$, $\vph(ws)=3$, $\vph(sv)=1$,
$\vph(su)=0$, and $\vph(st)=2$.  So assume $x$ sees 2 and misses 0.  Consider
the (0,2)-chains at $t$, $v$, and $x$, and let $P$ be the chain that ends at
$\infty$.  If $P$ starts at $x$, then recoloring $P$ reduces to the previous
case, where $x$ misses 2.  If $P$ starts at $v$, then we recolor $P$ and use
nearly the same coloring as above; the only difference is that we let
$\vph(vx)=0$ (rather than $\vph(vx)=2$).  So we assume that $P$ starts at $t$,
and recolor it.

Now consider the (0,1)-chains in $G-E(H)$ at $t$, $u$, $v$, $w$, $y$, and $z$.
If the chain at $t$ ends at $u$, $v$, $y$, or $z$, then we recolor it (and let
$\vph(vx)=0$ if it ends at $v$) and let $\vph(st)=1$.  So assume $t$ and $w$ are
(0,1)-linked.  Now consider the (0,1)-chain $P$ at $u$.  If $P$ ends at $z$,
then recolor it and let $\vph(su)=1$, $\vph(sv)=2$, $\vph(vy)=3$, $\vph(yw)=0$,
$\vph(ws)=3$, $\vph(st)=0$.  If $P$ ends at $v$, then the only difference is
that we also let $\vph(vx)=0$.  If $P$ ends at $y$, then nearly the
same idea works.  Now we recolor $P$, and let $\vph(su)=1$, $\vph(sv)=2$,
$\vph(vy)=3$, $\vph(yw)=2$, $\vph(wz)=0$, $\vph(ws)=3$, $\vph(st)=0$.  This
finishes the case when $\vph(vx)=1$.

Finally, assume that $\vph(vx)=0$; again recall that $y$ misses 0.  If
$\vph(vy)=1$, then $\vph(vv')=2$.  In this case, let $\vph(vy)=3$,
$\vph(yw)=0$, $\vph(ws)=3$, $\vph(sv)=1$, $\vph(su)=0$, and $\vph(st)=2$.  So
assume instead that $\vph(vy)=2$ and $\vph(vv')=1$.  
We show that if $x$ does not miss 2, then we can use a (1,2)-swap at $x$ (possibly
preceded by a (1,3)-swap at $x$) to assume that $x$ misses 2, as follows.
If a (1,3)-chain starts at $x$, then recoloring it cannot recolor edges incident
to $w$, since this would make $s$ and $u$ become (0,1)-unlinked.  The same is
true for a (1,2)-chain.  If the (1,2)-chain $P$ at $v$ in $G-vy$ does not end at
$t$ or $u$, then we recolor it and let $\vph(vy)=3$, $\vph(yw)=0$, $\vph(ws)=3$,
$\vph(sv)=1$, $\vph(su)=0$, and $\vph(st)=2$.  So $P$ must end at $t$ or $u$.
Since $t$ and $u$ are (1,2)-linked in $G$, the (1,2)-chain at $y$ in $G-vy$ also
ends at $t$ or $u$.  Thus, a (1,2)-swap at $x$ does not recolor edges incident
to $t$, $u$, $v$, $w$, $y$, or $z$.
Hence, by using a (1,2)- and (1,3)-swap at
$x$, we can assume that $x$ misses 2.  Consider the (0,1)-chains in $G-E(H)$ at
$u$, $v$, $w$, $x$, $y$, $z$.  Recall that $u$ and $w$ must be (0,1)-linked
(possibly through $v$ and $x$) or else we can recolor the (0,1)-chain at $u$ and
let $\vph(us)=1$ and $\vph(st)=2$.  So clearly neither $u$ nor $w$ is
(0,1)-linked to either $y$ or $z$.  Further, neither $u$ nor $w$ is (0,1)-linked
to $x$, since we can recolor that (0,1)-chain, let $\vph(vx)=2$, $\vph(vy)=0$,
and proceed as before.  So $u$ and $w$ must be (0,1)-linked.  Thus, $v$ is
(0,1)-linked to $x$, $y$, or $z$.  Let $P$ be the (0,1)-chain in $G-E(H)$ at
$v$.  If $P$ ends at $x$ or $z$, then recolor it and let $\vph(vx)=2$,
$\vph(vy)=3$, $\vph(yw)=0$, $\vph(ws)=3$, $\vph(sv)=1$, $\vph(su)=0$, and
$\vph(st)=2$.  So assume instead that $P$ ends at $y$.  Now recolor $P$ and let
$\vph(vx)=2$, $\vph(vy)=3$, $\vph(yw)=2$, $\vph(wz)=0$, $\vph(ws)=3$,
$\vph(sv)=1$, $\vph(su)=0$, and $\vph(st)=2$.  This completes Case 2a.

\textbf{Case 2b: we have Figure~\ref{lem3case2}(a).}
If a (1,3)-swap elsewhere ever recolors $wz$, then we can
finish as in the second paragraph of Case 2, when $\vph(wz)=1$.  So we assume
this never happens.  We show that we may assume all edges are colored as
in Figure~\ref{lem3case2b}.  After that, the proof is easy, as we show in the
final paragraph below.  First, we show that
$\vph(vx)=1$.  Suppose not.  By symmetry between $x$ and $y$, we assume 
$\vph(vx)=2$, $\vph(vy)=0$, and $\vph(vv')=1$. If $x$ misses 1, then after a
(1,2)-swap at $x$, we have $\vph(vx)=1$, as desired.  If $x$ misses 3 and sees
1, then a (1,3)-swap at $x$ gets $x$ missing 1.  So assume $x$ misses 0.
Note that $u$ and $t$ must be (0,3)-linked; otherwise we 
use a (0,3)-swap at $u$, followed by a (1,3)-swap at $u$, which results in $s$
and $t$ being (1,2)-unlinked, a contradiction.  So $y$ must see 3; otherwise a
(0,3)-swap at $x$ gives $x$ missing 3 (since the (0,3)-chain at $x$ does not
interact with other edges shown colored 0 or 3).  
Now $y$ also sees either 1 or 2.  We assume $y$ sees 2;
otherwise, we use a (1,2)-swap at $y$ (if this
recolors $vx$, then we have $\vph(vx)=1$, as desired; so
assume not).  The (0,1)-chain at $y$ must end at $z$; otherwise, we
recolor it, and have $\vph(vy)=1$, as desired.  So we can recolor the
(0,1)-chain at $x$ without effecting any other edges shown.  Now after a
(1,2)-swap at $x$, we have $\vph(vx)=1$, as desired.  Since $s$ and $t$ are
(1,3)-linked, $x$ also sees 3.  

Now we show that we may assume $\vph(vy)=2$ and $\vph(vv')=0$.  Assume to the
contrary that $\vph(vy)=0$ and $\vph(vv')=2$.  If $x$ misses 0 and $y$ misses
3, then a (0,1)-swap at $x$ makes $s$ and $t$ be (1,3)-unlinked, a
contradiction.  If $x$ misses 0 and $y$ misses 1, then a (1,3)-swap at
$y$ causes $y$ to miss 3 (as in the previous sentence).  So the four
possibilities for the ordered pair of colors missed at $x$ and $y$
are (0,2), (2,1), (2,2), (2,3).
We reduce to the cases (2,2) and (2,3), as follows.  In the
case (2,1), a (1,3)-swap at $y$ yields (2,3), as desired.  Suppose we are
in the case (0,2), and use a (1,2)-swap at $y$.  This must recolor the path
through $vx$; otherwise we are in the case (0,1), handled above.  Now the
(0,1)-chain at $y$ must end at $z$, or we recolor it.  So we can
use a (0,1)-swap at $x$.  Now (1,2)-swaps at $x$ and $y$ yield the case (2,2).
So it suffices to handle the cases (2,2) and (2,3).

\begin{figure}
\centering
\begin{tikzpicture}[scale = 10]
\tikzstyle{VertexStyle} = []
\tikzstyle{EdgeStyle} = []
\tikzstyle{labeledStyle}=[shape = circle, minimum size = 14pt, inner sep = 1.2pt, draw]
\tikzstyle{unlabeledStyle}=[shape = circle, minimum size = 6pt, inner sep = 1.2pt, draw]
\Vertex[style = labeledStyle, x = 0.600, y = 0.850, L = \footnotesize {$s$}]{v0}
\Vertex[style = labeledStyle, x = 0.300, y = 0.650, L = \footnotesize {$t$}]{v1}
\Vertex[style = labeledStyle, x = 0.500, y = 0.650, L = \footnotesize {$u$}]{v2}
\Vertex[style = labeledStyle, x = 0.700, y = 0.650, L = \footnotesize {$v$}]{v3}
\Vertex[style = labeledStyle, x = 0.900, y = 0.650, L = \footnotesize {$w$}]{v4}
\Vertex[style = labeledStyle, x = 0.600, y = 0.450, L = \footnotesize {$x$}]{v5}
\Vertex[style = labeledStyle, x = 0.800, y = 0.450, L = \footnotesize {$y$}]{v6}
\Vertex[style = labeledStyle, x = 1.000, y = 0.450, L = \footnotesize {$z$}]{v7}
\Vertex[style = unlabeledStyle, x = 0.250, y = 0.550, L = \tiny {}]{v8}
\Vertex[style = unlabeledStyle, x = 0.350, y = 0.550, L = \tiny {}]{v9}
\Vertex[style = unlabeledStyle, x = 0.450, y = 0.550, L = \tiny {}]{v10}
\Vertex[style = unlabeledStyle, x = 0.550, y = 0.550, L = \tiny {}]{v11}
\Vertex[style = unlabeledStyle, x = 0.600, y = 0.650, L = \tiny {}]{v12}
\Vertex[style = unlabeledStyle, x = 1.050, y = 0.700, L = \tiny {}]{v13}
\Vertex[style = unlabeledStyle, x = 1.050, y = 0.650, L = \tiny {}]{v14}
\Vertex[style = unlabeledStyle, x = 0.550, y = 0.350, L = \tiny {}]{v15}
\Vertex[style = unlabeledStyle, x = 0.650, y = 0.350, L = \tiny {}]{v16}
\Vertex[style = unlabeledStyle, x = 0.750, y = 0.350, L = \tiny {}]{v17}
\Vertex[style = unlabeledStyle, x = 0.850, y = 0.350, L = \tiny {}]{v18}
\Vertex[style = unlabeledStyle, x = 0.950, y = 0.350, L = \tiny {}]{v19}
\Vertex[style = unlabeledStyle, x = 1.050, y = 0.350, L = \tiny {}]{v20}
\Edge[label = \footnotesize {$3$}, labelstyle={xshift=.08in, fill=none}](v3)(v0)
\Edge[label = \footnotesize {$0$}, labelstyle={xshift=.15in, fill=none}](v4)(v0)
\Edge[label = \footnotesize {$0$}, labelstyle={xshift=-.08in, fill=none}](v1)(v8)
\Edge[label = \footnotesize {$1$}, labelstyle={xshift=.08in, fill=none}](v9)(v1)
\Edge[label = \footnotesize {$3$}, labelstyle={xshift=.08in, fill=none}](v11)(v2)
\Edge[label = \footnotesize {$2$}, labelstyle={xshift=.08in, fill=none}](v20)(v7)
\Edge[label = \footnotesize {$1$}, labelstyle={xshift=-.08in, fill=none}](v2)(v10)
\Edge[label = \footnotesize {$2$}, labelstyle={xshift=-.08in, fill=none}](v0)(v2)
\Edge[label = \footnotesize {$1$}, labelstyle={xshift=-.08in, fill=none}](v7)(v19)
\Edge[label = \footnotesize {$3$}, labelstyle={xshift=.08in, fill=none}](v7)(v4)
\Edge[label = \footnotesize {$2$}, labelstyle={yshift=-.08in, fill=none}](v4)(v14)
\Edge[label = \footnotesize {}, labelstyle={auto=right, fill=none}](v1)(v0)
\Edge[label = \footnotesize {$2$}, labelstyle={xshift=.08in, fill=none}](v6)(v3)
\Edge[label = \footnotesize {$1$}, labelstyle={yshift=.08in, fill=none}](v13)(v4)
\Edge[label = \footnotesize {$2$}, labelstyle={xshift=.08in, fill=none}](v16)(v5)
\Edge[label = \footnotesize {$1$}, labelstyle={xshift=.08in, fill=none}](v18)(v6)
\Edge[label = \footnotesize {$1$}, labelstyle={xshift=-.08in, fill=none}](v3)(v5)
\Edge[label = \footnotesize {$3$}, labelstyle={xshift=-.08in, fill=none}](v5)(v15)
\Edge[label = \footnotesize {$0$}, labelstyle={xshift=-.08in, fill=none}](v6)(v17)
\Edge[label = \footnotesize {$0$}, labelstyle={yshift=.08in, fill=none}](v3)(v12)
\end{tikzpicture}
\caption{A coloring of $G-st$ in Case 2 of the proof of
Lemma~\ref{lem3}.\label{lem3case2b}}
\end{figure}
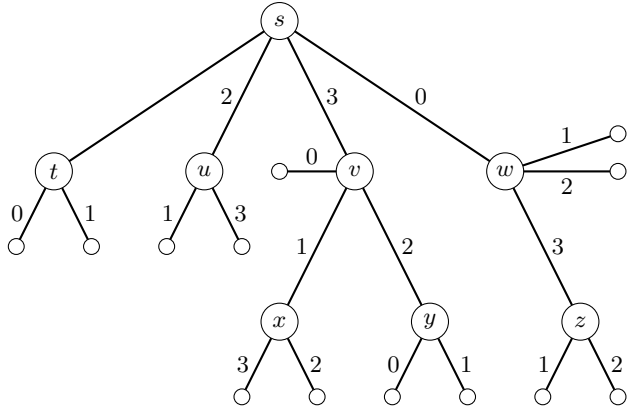

Suppose we are in the case (2,2); that is, both $x$ and $y$ miss 2.
Use (1,2)-swaps at $x$ and $y$, followed by (1,3)-swaps at $x$
and $y$.  Now at $x$ we use a (0,3)-swap; this cannot recolor edges incident to
$t$ or $u$, since $t$ and $u$ must be (0,3)-linked, as in the first paragraph of
Case 2b.  Consider the (0,1)-chain $P$ at $v$ in $G-vy$.  If $P$ ends at $u$,
$x$, $y$, or $z$, then recolor $P$ and let $\vph(vy)=3$, $\vph(vs)=1$, and
$\vph(st)=3$.  So assume $P$ ends at $w$.  Recall that $s$ and $u$ must be
(0,1)-linked in $G$.  Since $P$ ends at $w$, the (0,1)-chain $Q$ at $u$ must
end at $y$.  Recolor $Q$ and let $\vph(vy)=3$, $\vph(vx)=0$, $\vph(vs)=2$,
$\vph(su)=1$, and $\vph(st)=3$.

Finally, assume we are in the case (2,3).
Consider the (0,1)-chains in $G-E(H)$ starting at $u$, $w$, $x$, $y$, $z$.
Let $P$ be the chain starting at $x$.  
If $P$ does not end at $w$, then recolor $P$, and let $\vph(vx)=0$,
$\vph(vy)=3$, $\vph(vs)=1$, and $\vph(st)=3$.
So $P$ must end at $w$.  Let $Q$ be the (0,1)-chain starting at $u$.  If $Q$
ends at $z$ or $\infty$, then recolor $Q$ and let $\vph(us)=1$ and
$\vph(st)=2$.  
So $Q$ must end at $y$.  Now recolor $Q$, and let $\vph(yv)=3$,
$\vph(vs)=0$, $\vph(sw)=3$, $\vph(wz)=0$, $\vph(us)=1$, $\vph(st)=2$.
Thus, we conclude that $\vph(vy)=2$, and so $\vph(vv')=0$.

Now we need only to show that the colors missing at $x$ and $y$ are as in
Figure~\ref{lem3case2b}.  Since $x$ and $t$ are (1,3)-linked, $x$ sees 3,
so we have 6 possibilities for these missing colors.  If $x$ misses 2
and $y$ misses 3, then a (1,2)-swap at $x$ makes $s$ and $t$ 
(1,3)-unlinked.  If $x$ misses 2 and $y$ misses 1, then a (1,3)-swap at $y$
takes us to the previous sentence.  So suppose $x$ misses 2 and $y$
misses 0.  Now a (1,2)-swap at $x$ (and interchanging the roles of $x$ and $y$)
yields the case that $x$ misses 0 and $y$ misses 1.  Thus, we assume that
$x$ misses 0.  Suppose that $y$ misses 0.  Now $x$ must be (0,3)-linked to $v$;
otherwise a (0,3)-swap at $x$ makes $s$ and $t$ be (1,3)-unlinked, a
contradiction.  Now we (0,3)-swap at $y$, after which $y$ misses 3.  This
gives the colors in Figure~\ref{lem3case2b}.  So assume $x$ misses 0 and $y$
misses 1.  Now we (1,3)-swap at $y$, which again gives the colors in
Figure~\ref{lem3case2b}.

Finally, we show that if the colors are as in Figure~\ref{lem3case2b}, then $G$
has a coloring.  Consider the (1,2)-chains in $G-E(H)$ that start at
$t$, $u$, $x$, $y$.  If the (1,2)-chain at $y$ ends at $x$, then recolor it and
let $\vph(vx)=2$ and $\vph(vy)=1$.  Now $s$ and $t$ are (1,3)-unlinked, a
contradiction.  If the (1,2)-chain at $y$ ends at $u$, then recolor it and let
$\vph(yv)=3$, $\vph(vs)=2$, $\vph(su)=1$, $\vph(st)=3$.  So assume the
(1,2)-chain at $y$ ends at $t$. 
Recolor it and let $\vph(yv)=3$, $\vph(vs)=2$, 
$\vph(sw)=3$, $\vph(wz)=0$, $\vph(su)=0$, $\vph(st)=1$.  
This completes Case 2b, and the proof. 
\end{proof}

\section*{Acknowledgments}
The first author thanks Beth Cranston for encouragement and support during
both the research and writing phases of this paper.
Thanks also to three anonymous referees; one made suggestions that led to numerous
improvements.

\bibliographystyle{siam}
{\footnotesize{\bibliography{GraphColoring}}
\end{document}